\documentclass{amsart} 
\usepackage{pgf,tikz,hyperref,comment}
\usepackage{amsmath, graphicx}
\usepackage{amssymb,color,xcolor}
\usepackage{enumitem}
\usepackage{longtable}

\numberwithin{equation}{section}

\newtheorem{theorem}{Theorem}[section]
\newtheorem{proposition}[theorem]{Proposition}
\newtheorem{lemma}[theorem]{Lemma}
\newtheorem{corollary}[theorem]{Corollary}

\newtheorem*{theorem*}{Theorem}

\newtheorem{assump}{Theorem}

\theoremstyle{definition}

\newtheorem{prob*}{Problem}

\newtheorem{definition}[theorem]{Definition}

\newtheorem{example}[theorem]{Example}
\newtheorem{remark}[theorem]{Remark}

\newtheorem{thevarthm}[theorem]{\varthmname}

\newenvironment{varthm*}[1]{\trivlist\item[]{\bf #1.}\it}{\endtrivlist}

\newcommand{\PP}{ \ensuremath{\mathbb{P}}}

\newcommand{\field}{K}
\newcommand{\acfield}{\overline{\field}}
\newcommand{\calm}{\mathcal M}

\begin{document}

\author[Chiantini]{Luca Chiantini}
\address[L.~Chiantini]{Dipartmento di Ingegneria dell'Informazione e Scienze Matematiche, Universit\`a di Siena, Italy}
\email{luca.chiantini@unisi.it}

\author[Farnik]{{\L}ucja Farnik}
\address[{\L}.~Farnik]{Department of Mathematics, University of the National Education Commission, Krakow,
   Podcho\-r\c a\.zych~2,
   PL-30-084 Krak\'ow, Poland}
\email{lucja.farnik@gmail.com}

\author[Favacchio]{Giuseppe Favacchio}
\address[G.~Favacchio]{Dipartimento di Ingegneria, Universit\`a degli studi di Palermo,
Viale delle Scienze,  90128 Palermo, Italy}
\email{giuseppe.favacchio@unipa.it}

\author[Harbourne]{Brian Harbourne}
\address[B.~Harbourne]{Department of Mathematics,
University of Nebraska,
Lincoln, NE 68588-0130 USA}
\email{brianharbourne@unl.edu}

\author[Migliore]{Juan Migliore} 
\address[J.~Migliore]{Department of Mathematics,
University of Notre Dame,
Notre Dame, IN 46556 USA}
\email{migliore.1@nd.edu}

\author[Szemberg]{Tomasz Szemberg}
\address[T.~Szemberg]{Department of Mathematics, University of the National Education Commission, Krakow,
   Podcho\-r\c a\.zych~2,
   PL-30-084 Krak\'ow, Poland}
\email{tomasz.szemberg@gmail.com}

\author[Szpond]{Justyna Szpond}
\address[J.~Szpond]{Department of Mathematics, University of the National Education Commission, Krakow,
   Podcho\-r\c a\.zych~2,
   PL-30-084 Krak\'ow, Poland}
\email{szpond@gmail.com}

\thanks{Chiantini and Favacchio are members of the Italian GNSAGA-INDAM. The work of the Favacchio was supported by the funding PREMIO\_SINGO\-LI\_RIC\_[2025] from the Department of Engineering, University of Palermo. Migliore was partially supported by Simons Foundation grant \#839618.}

\title[Enumerative geometry of skew lines in $\PP^3$]{Enumerative geometry of skew lines in $\PP^3$ with a given associated finite group}

\begin{abstract}
For any finite set $\mathcal L$ of 3 or more skew lines in $\PP^3_{\acfield}$ over an algebraically closed field $\acfield$ of arbitrary characteristic, there is a canonical associated subgroup $G_{\mathcal L}$ of ${\rm PGL}_2(\acfield)$. Given a finite subgroup $G\subset{\rm PGL}_2(\acfield)$ we study 
 which configurations of lines have $G_{\mathcal L}=G$. We derive an upper bound on the number $|\mathcal L|$ of lines in terms of the order $|G|$ of the group $G$ and as an application we
classify up to projective equivalence which sets $\mathcal L$ in $\PP^3_{\mathbb C}$ have $G_{\mathcal L}=G$ for certain finite nonabelian groups $G$.
\end{abstract}



\date{July 3, 2026}

\maketitle

\section{Introduction}\label{Intro}

This paper establishes a new
problem in enumerative geometry and shows how work on
classifying sets of skew lines in $\PP^3_\field$ over finite fields $K$ (these are called \emph{spreads} in the literature) can be carried over to any field, including the complex numbers.

Let $K$ be a field and let $\overline K$ denote its algebraic closure.
The basic problem is to classify, up to projective equivalence, finite sets of skew lines in $\PP^3_{\field}$ over a field $\field$. 
This is  a finite classification problem when $\field$ is finite,
since there are only finitely many lines defined by pairs of linear forms in 4 variables over a finite field. Whereas, over any field, there is only one projective equivalence class of sets of 3 skew lines, over $\mathbb C$ there are, even for sets of only 4 skew lines, infinitely many projective equivalence classes.
(Given a set $\mathcal L$ of 4 skew lines in $\PP^3_\field$, 
there is always a line $T\subset\PP^3_{\acfield}$ meeting all of the lines in $\mathcal L$. The fact
over $\mathbb C$ that there are infinitely many projective equivalence classes of 4 skew lines 
is related to
the fact that there are infinitely many possible values of the cross ratio of the points of 
intersection of $T$ with the 4 lines in $\mathcal L$.)

However, as we explain in more detail below, 
for each finite set $\mathcal L$ of 3 or more skew lines in $\PP^3_{\acfield}$,
there is an associated group
$G_{\mathcal L}$ which is canonically a subgroup
of ${\rm Aut}(L)$ for each $L\in\mathcal L$.
(The group $G_{\mathcal L}$ is not finite in general, but 
it is finite if $\acfield$ is the algebraic closure of a finite field.) 
If we fix three skew lines (which we refer to below as $L_\infty, L_0$ and $L_1$) and consider only
sets $\mathcal L$ containing $L_\infty, L_0$ and $L_1$, then with respect to a fixed isomorphism 
${\rm Aut}(L)\cong {\rm PGL}_2(\acfield)$
and a fixed finite subgroup $G\subset {\rm PGL}_2(\acfield)$
our main result, Theorem \ref{mainThm}, shows that there is a unique minimal finite set of lines $\calm_G$ in $\PP^3_{\acfield}$ such that whenever $G_{\mathcal L}=G$ we always have
$\mathcal L\subseteq \calm_G$.

This shows that the well-studied problem of classifying spreads (i.e., sets of skew lines)
over finite fields can be carried over to any field
simply by restricting to lines with a given  finite group $G\subset {\rm PGL}_2(\acfield)$. The problem then becomes to classify subsets $\mathcal L'$ of the finite set $\calm_G$ such that $G_{\mathcal L'}=G$.
One can ask all of the same questions as is done with spreads over finite fields,
but now in the context of having a specified group $G$: 
what is the least number of lines for which we get a given group  $G$, what is the maximum 
number, which subsets of lines are maximal, which are minimal, how many subsets of lines of 
given cardinality are there up to projective equivalence that have the given group $G$. 

Given a finite set $\mathcal L=\{L_1,\ldots,L_s\}$ of $s\ge 3$ pairwise skew lines in
$\PP^3_{\overline K}$, the paper \cite{POLITUS3} associates to
$\mathcal L$ a groupoid $\mathcal G_{\mathcal L}$  whose objects are the lines of $\mathcal L$. (Recall that a groupoid is a category for which the morphisms are all invertible.) The basic
morphisms are obtained from triples of skew lines: if $L_i,L_j,L_k$ are
distinct and $p\in L_i$, the plane spanned by $p$ and $L_k$ meets $L_j$ in a
unique point $q$, giving an isomorphism
$f_{ijk}\colon L_i\longrightarrow L_j$
with $f_{ijk}(p)=q$.
Taking all such maps and their compositions, the maps $L_i\to L_j$ that we get in this way 
comprise the Hom set ${\rm Hom}_{\mathcal G_{\mathcal L}}(L_i,L_j)$ of the groupoid $\mathcal G_{\mathcal L}$ associated to
$\mathcal L$.

For each line $L_i\in\mathcal L$, the automorphisms ${\rm Hom}_{\mathcal G_{\mathcal L}}(L_i,L_i)$ of $L_i$ in this groupoid form a subgroup of
\[
{\rm  Aut}(L_i)\cong {\rm PGL}_2(\overline K).
\]
These groups are naturally isomorphic for different choices of $L_i$, and we denote this common group by $G_{\mathcal L}$. Thus $G_{\mathcal L}$ is a projective invariant of the configuration of skew lines. The purpose of this note is to study how many different configurations ${\mathcal L}$ have isomorphic  finite groups $G_{\mathcal L}$. Of particular interest is the case that $G_{\mathcal L}$ is nonabelian, since \cite{POLITUS3} has given a good foundation for understanding the abelian case. By \cite{Favacchio}, if $\field=\mathbb C$ and 
$G_{\mathcal L}$ is finite nonabelian,
then $G_{\mathcal L}$ is either $A_4$,
$S_4$ or $A_5$.

We will use the matrix description of $G_{\mathcal L}$ developed in \cite{Favacchio}. Fix coordinates
$x,y,z,w$ on $\PP^3_\field$ and set
\[
L_\infty=V(x,y),\qquad L_0=V(z,w).
\]
For a matrix
\[
M=\begin{pmatrix}a&b\\ c&d\end{pmatrix}\in \mathrm{GL}_2(K),
\]
we denote by $L_M\subset \PP^3_\field$ the line defined by
\[
z=ax+by,\qquad w=cx+dy.
\]
Thus $L_{I_2}$ is the line defined by $z=x$ and $w=y$.

Since any three skew lines in $\PP^3_\field$ are projectively equivalent,
after a projective change of coordinates we may assume that three of the lines
in our configuration are
$L_\infty, L_0, L_{I_2}$.
Every further line disjoint from $L_\infty$ and $L_0$ is then uniquely of the
form $L_M$ for some $M\in \mathrm{GL}_2(K)$. Moreover, two such lines $L_M$ and
$L_N$ are skew if and only if $M-N$ is invertible.

In these coordinates, if
\[
\mathcal L=\{L_\infty,L_0,L_{M_1},L_{M_2},\ldots,L_{M_r}\},
\qquad M_1=I_2,
\]
then \cite{Favacchio} shows that the group $G_{\mathcal L}$ is generated in
${\rm PGL}_2(\overline K)$ by the classes $[M_i]$ and $[M_i-M_j]$ modulo scalars of the matrices
\[
M_i \quad\text{for}\quad  1\le i\le r
\qquad\text{and}\qquad
M_i-M_j \quad\text{for}\quad 1\le i<j\le r.
\]
This gives a concrete way to compute and constrain $G_{\mathcal L}$.

A notable feature of this description is that the generators of
$G_{\mathcal L}$ are not obtained only from products of the matrices $M_i$.
They also involve the differences $M_i-M_j$. At first sight this is somewhat
unexpected from a group-theoretic point of view, where one usually expects
closure under multiplication only. Here the appearance of differences is related to
the geometry: the condition that the lines $L_{M_i}$ and $L_{M_j}$ are skew is
equivalent to the invertibility of $M_i-M_j$. Thus each such difference defines
a well-defined element of $\mathrm{PGL}_2(K)$, which arises in defining the maps $f_{ijk}$ \cite{Favacchio}.

We are interested in the natural realization problem: for a fixed finite
subgroup $G\subset {\rm PGL}_2(\overline K)$, describe all
configurations of skew lines $\mathcal L$ for which $G_{\mathcal L}\cong G$.
The matrix description above makes this a concrete problem about finite sets of
matrices whose classes, together with the classes of their differences, generate
the prescribed group.

As a first simple case, scalar matrices $cI_2$, $c\in\acfield$, $c\neq0$, give only the trivial group. 
Indeed, given $L_\infty, L_0$, $L_{M_1},\ldots,L_{M_r}$ for distinct scalar matrices $M_i=\lambda_i I_2$ with 
$\lambda_1=1$, the matrices $M_i$ and their differences $M_j-M_i$, $i<j$, all map to the 
identity in ${\rm PGL}_2(\acfield)$, hence generate the identity subgroup. Geometrically, each $L_{M_i}$ is given by 
$z=\lambda_i x, w=\lambda_i y$ and hence belongs to the same ruling of the quadric $xw-yz=0$, as do $L_\infty$ 
and $L_0$. Conversely, if a finite set of skew lines is contained in one ruling of a smooth quadric, 
the group is trivial, thereby giving a geometric criterion for triviality \cite{POLITUS3}.

The case of four skew lines in $\PP^3_{\mathbb C}$ was treated explicitly in
\cite[Corollary~2.1.24, Theorem~3.5.9]{POLITUS3}: after fixing three skew lines
$L_1,L_2,L_3$ and an integer $m>2$, there are, up to projective equivalence,
only finitely many fourth lines $L_4$ such that the associated group $G_{\mathcal L}$ for
$\mathcal L=\{L_1,L_2,L_3,L_4\}$ has order $m$; in this case $G_{\mathcal L}$ is cyclic.

However, classifying sets of 3 or more skew lines $\mathcal L$ in $\PP^3_{\mathbb C}$ where $G_{\mathcal L}$ is finite abelian heavily used the fact that $G_{\mathcal L}$ is abelian if and only if there are two lines $T_1,T_2$ which each meet every line in $\mathcal L$. The lack of these transversals in the nonabelian case has resulted, up to now,
in not having methods to classify skew lines whose group is finite but nonabelian. Our main result, Theorem \ref{mainThm}, now provides the basis for classification in the nonabelian case (although the theorem applies and is of interest also in the abelian case).
We demonstrate this by showing,
up to projective equivalence, that there is a
unique set of skew lines $\mathcal L$ 
in $\PP^3_{\mathbb C}$ with $G_{\mathcal L}\cong A_4$ and we
find, up to projective equivalence, all sets of skew lines $\mathcal L$ in $\PP^3_{\mathbb C}$ with 
$G_{\mathcal L}\cong S_4$.
In the case of $A_4$, this unique set of lines consists of 5 lines.  
In the case of $S_4$, there is not a unique set of lines, but we find that every set of 3 or more lines with group $S_4$ has at least 5 and at most 10 lines.

We close this introduction with a few words of motivation.
This work arose from studying finite sets of points 
$Z\subset \PP^3_{\acfield}$ whose image $\overline{Z}$ under projection from a general point
to a plane is a complete intersection. Such sets are called geproci sets,
for ``GEneral PROjection is a Complete Intersection.''
In \cite{POLITUS3} it is shown that a major class of geproci sets
are groupoid orbits of points on finite sets of skew lines, but \cite{POLITUS3} left largely open the question of which sets of lines have nonabelian  associated finite groups, and thus left open the problem of classifying geproci sets which come from lines with nonabelian groups.

\section{Preliminaries}
We keep the notation used in the introduction. 
We begin by recalling the basic properties of this matrix description.
The next lemma is a standard way to parametrize the lines in $\PP^3$ disjoint from two fixed skew lines by $2\times 2$ matrices \cite[Lecture 6]{HarrisAG}. We include an elementary proof for completeness, since the criterion involving $M_1-M_2$ will be used throughout.

\begin{lemma}\label{MatrixLineAssoc}
If $L\subset \PP^3_{\acfield}$ is a line disjoint from $L_0$ and $L_\infty$, then $L=L_M$ 
for a unique matrix $M\in {\rm GL}_2(\acfield)$. 
Conversely, if $M\in {\rm GL}_2(\acfield)$, then $L_M$ is disjoint from $L_0$ and $L_\infty$.
Moreover, if $M_1, M_2\in {\rm GL}_2(\acfield)$, then $L_{M_1}$
and $L_{M_2}$ are disjoint if and only if $M_1-M_2$ is invertible.
\end{lemma}
 
\begin{proof}
Every line $L$ is defined by linearly independent equations $$Cz+Dw+Ax+By~=~0,\qquad C'z+D'w+A'x+B'y~=~0.$$
We can write this as the matrix equation 
$$
\begin{pmatrix}
C & D & A & B\\
C' & D' & A' & B'\\
\end{pmatrix}
\begin{pmatrix}
z\\
w\\
x\\
y\\
\end{pmatrix}
=0.
$$
Set 
$$P=\begin{pmatrix}
C & D\\
C' & D'
\end{pmatrix}.$$
The assumption that $L$ is disjoint from $L_\infty$ implies $P$ must have rank $2$.
	Indeed, if $P$ had rank less than $2$, there would be a nonzero vector
	$(z_0,w_0)$ satisfying
$P\begin{pmatrix}z_0\\w_0\end{pmatrix}=0,$
	and hence a point $[0:0:z_0:w_0]\in L\cap L_\infty$. 

Since $P$ has rank $2$, we can multiply the matrix equation
defining $L$ on the left by
	$-P^{-1}$. This gives a matrix equation for $L$ equivalent to
	\[
	z=ax+by,\qquad w=cx+dy,
	\]
    for some scalars $a,b,c,d$,
	which shows $L=L_M$ for the matrix $M=\begin{pmatrix}a&b\\ c&d\end{pmatrix}$.
The assumption that $L$ is disjoint from $L_0$ forces $M$ to be invertible.
	Indeed, a point of $L\cap L_0$ would have the form $[x_0:y_0:0:0]$ with
	$(x_0,y_0)\neq (0,0)$ and would satisfy $M\begin{pmatrix}x_0\\y_0\end{pmatrix}=0.$
To see that $M$ is unique, assume $L_M=L_N$ for invertible matrices $M$ and $N$.
Then for each $x_0$ and $y_0$ we get a unique point $[x_0:y_0:z_0:w_0]\in L_M$ where 
$$
\begin{pmatrix}
z_0\\
w_0\\
\end{pmatrix}=
M\begin{pmatrix}
x_0\\
y_0\\
\end{pmatrix}.$$
Since $[x_0:y_0:z_0:w_0]\in L_N$ we also have 
$$
\begin{pmatrix}
z_0\\
w_0\\
\end{pmatrix}=
N\begin{pmatrix}
x_0\\
y_0\\
\end{pmatrix}\qquad
\text{and hence}\qquad
M\begin{pmatrix}
x_0\\
y_0\\
\end{pmatrix}=
N\begin{pmatrix}
x_0\\
y_0\\
\end{pmatrix}$$
for all $x_0$ and $y_0$, so $M=N$.

Conversely, if $M$ is an invertible matrix, then $L_M\cap L_0=\emptyset$ and $L_M\cap L_\infty=\emptyset$
since on $L_M$ we have $z=w=0$ if and only if $x=y=0$.

The last claim is \cite[Lemma 2.2]{Favacchio}. For convenience we include the proof.
A point $[u_0:u_1:u_2:u_3]$ is in $L_{M_i}$ if and only if
$M_i\begin{pmatrix} u_0\\ u_1\end{pmatrix}=\begin{pmatrix} u_2\\ u_3\end{pmatrix}$
(in which case $\begin{pmatrix} u_0\\ u_1\end{pmatrix}\neq \begin{pmatrix} 0\\ 0\end{pmatrix}$).
Thus $[u_0:u_1:u_2:u_3]\in L_{M_1}\cap L_{M_2}$ if and only if $(M_1-M_2)\begin{pmatrix} u_0\\ u_1\end{pmatrix}=0$,
hence if and only if $M_1-M_2$ is not invertible.
\end{proof}

\begin{definition}\label{def:GL-matrix}
	Let $\mathcal L$ be a finite set of at least three skew lines in
	$\PP^3_{\acfield}$. After a projective change of coordinates, we may assume
	that $L_\infty,   L_0,  L_{I_2}$
	belong to $\mathcal L$. Thus we may write
	\[
	\mathcal L=\{L_\infty,L_0,L_{M_1},L_{M_2},\ldots,L_{M_r}\},\qquad M_1=I_2
	\]
	where $M_i\in {\rm GL}_2(\acfield)$ and $M_i-M_j$ is invertible for all
	$i\neq j$.
	
	We define $G_{\mathcal L}$ to be the subgroup of
	${\rm PGL}_2(\acfield)$ generated by the classes of the matrices
	\[
M_i \quad\text{for}\quad  1\le i\le r
\qquad\text{and}\qquad
M_i-M_j \quad\text{for}\quad 1\le i<j\le r.
\]
\end{definition}

\begin{remark}\label{IndependentRem}
	The group $G_{\mathcal L}$ defined above agrees with the group associated to
	$\mathcal L$ in \cite{POLITUS3}. In \cite{POLITUS3}, for any $L\in \mathcal L$, the group is defined 
	as ${\rm Hom}_{\mathcal G_{\mathcal L}}(L,L)\subset {\rm Aut}(L)$ for
	the groupoid $\mathcal G_{\mathcal L}$ (see \cite{POLITUS3}, \cite{Favacchio} and Remark \ref{rem:GroupoidForS_4}). 
    The groupoid, and hence the
	containment ${\rm Hom}_{\mathcal G_{\mathcal L}}(L,L)\subset {\rm Aut}(L)$,
	is independent of any choice of coordinates on $\PP^3_{\acfield}$. The matrix
	description used here is the one developed in \cite{Favacchio}, and gives an
	explicit set of generators for the same group.
\end{remark}

\begin{example}\label{S4Example}
Here we give an example of a set $\mathcal L$ of 10 skew lines in $\PP^3$ over the complex numbers. Their group $G_{\mathcal L}$ is $S_4$. 
The 10 lines consist of $L_\infty$, $L_0$ and $L_1=L_{M_1}=L_{I_2}$, and in addition, lines
$L_{M_2},\ldots,L_{M_8}$ where $M_i\in {\rm GL}_2(\mathbb C)$ for each $i$. Each matrix maps, modulo scalar matrices, to an element of a fixed
$S_4\subset {\rm PGL}_2(\mathbb C)$.
(We write $M\equiv N$ to denote that $M$ and $N$ are equivalent modulo scalar matrices.) Below we give a list of representatives $U_i$ of this $S_4$.
The first four elements give the Klein four group,
the first 12 give $A_4$, and the last 12 map to
odd permutations in $S_4$. The last twelve are of the form $U_{13}U_i$ for $1\leq i\leq 12$
(except we take $U_{14}=-U_{13}U_2$).
Note that $U_i$ for $i=13,14,17,20,22,23$
map to elements of $S_4$ of order 4, while the $U_i$ for
$i=15,16,18,19,21,24$ map to 2-cycles.
We will revisit the $U_i$ later when we analyze lines having group $A_4$:

{\footnotesize
$$
\renewcommand{\arraystretch}{1.35}
\begin{array}{l}
U_1=I_2, \
U_2^{-1}\equiv U_2=\begin{pmatrix}0&-1\\ 1&0\end{pmatrix},\ 
U_3^{-1}\equiv U_3=\begin{pmatrix}t&t^*\\ t^*&-t\end{pmatrix},\ 
U_4^{-1}\equiv U_4=\begin{pmatrix}-t^*&t\\ t&t^*\end{pmatrix}, \\

U_5=\begin{pmatrix}t&1\\ 0&t^*\end{pmatrix}, \
U_5^{-1}\equiv U_6=\begin{pmatrix}t^*&-1\\ 0&t\end{pmatrix}, \
U_7=\begin{pmatrix}t^*&0\\ -1&t\end{pmatrix}, \
U_7^{-1}\equiv U_8=\begin{pmatrix}t&0\\ 1&t^*\end{pmatrix}, \\

U_9=\begin{pmatrix}1&-t\\ t^*&0\end{pmatrix}, \
U_9^{-1}\equiv U_{10}=\begin{pmatrix}0&t\\ -t^*&1\end{pmatrix}, \
U_{11}=\begin{pmatrix}0&-t^*\\ t&1\end{pmatrix}, \
U_{11}^{-1}\equiv U_{12}=\begin{pmatrix}1&t^*\\ -t&0\end{pmatrix},\\

U_{13}=\begin{pmatrix}1&-1\\ 1&1\end{pmatrix}, \ 
U_{13}^{-1}\equiv U_{14}=\begin{pmatrix}1&1\\ -1&1\end{pmatrix}, \ 
U_{15}^{-1}\equiv U_{15}=\begin{pmatrix}2t-1&1\\ 1&-2t+1\end{pmatrix}, \\

U_{16}^{-1}\equiv U_{16}=\begin{pmatrix}-1&2t-1\\ 2t-1&1\end{pmatrix}, \ 
U_{17}=\begin{pmatrix}t&t\\ t&-t+2\end{pmatrix}, \ 
U_{18}^{-1}\equiv U_{18}=\begin{pmatrix}t^*&-t-1\\ t^*&-t^*\end{pmatrix}, \\

U_{17}^{-1}\equiv U_{19}=\begin{pmatrix}-t+2&-t\\ -t&t\end{pmatrix}, \ 
U_{20}^{-1}\equiv U_{20}=\begin{pmatrix}-t^*&-t^*\\ t+1&t^*\end{pmatrix}, \
U_{21}^{-1}\equiv U_{21}=\begin{pmatrix}t&-t\\ -t+2&-t\end{pmatrix}, \\

U_{22}=\begin{pmatrix}t^*&-t^*\\ -t^*&t+1\end{pmatrix}, \
U_{23}^{-1}\equiv U_{23}=\begin{pmatrix}-t&t-2\\ t&t\end{pmatrix}, \ 
U_{22}^{-1}\equiv U_{24}=\begin{pmatrix}t+1&t^*\\ t^*&t^*\end{pmatrix}.
\end{array}
$$
}
Here, $t$ and $t^*$ are the roots of $x^2-x+1=0$, hence $tt^*=t+t^*=1$.

These matrices modulo scalars are distinct and closed under multiplication so form a subgroup of ${\rm PGL}_2(\mathbb C)$ of order 24. It is nonabelian (for example, $U_5U_8\not\equiv U_8U_5$).
It cannot be a dihedral group by \cite[Theorem 4.3]{Favacchio}, so it must be $S_4$ by \cite[Theorem 7.1]{Faber}. 

The 10 lines for this example come from \cite[Example 3.2.10 (3)]{POLITUS3}. After normalizing with a suitable change of coordinates, the lines are $\mathcal L=\{L_\infty,L_0,L_{M_1},\ldots,L_{M_8}\}$,
where 
$M_1=I_2,M_2=\frac12 I_2,M_3=\frac12U_{13},M_4=\frac12U_{14},
M_5=\frac12U_{17},M_6=\frac12U_{19},M_7=\frac12U_{22}$ and $M_8=\frac12U_{24}$.

The group $G_{\mathcal L}$ is generated by the images of $M_1,\ldots,M_8$ and $M_{j,i}=M_j-M_i$ for $1\leq i<j\leq 8$. We have already seen that $M_i\equiv U_j$,
modulo scalars, for some $j$ for each $i$. Modulo scalars, we also have 
$M_2-M_1\equiv I_2$,
$M_4-M_3\equiv M_4-M_2\equiv M_3-M_2\equiv M_3-M_1=U_2$,
$M_4-M_1\equiv U_{13}$,
$M_5-M_1\equiv U_{19}$,
$M_6-M_1\equiv U_{17}$,
$M_8-M_7\equiv M_8-M_2\equiv M_7-M_2\equiv M_7-M_1\equiv U_3$,
$M_8-M_1\equiv U_{22}$,
$M_6-M_5\equiv M_6-M_2\equiv M_5-M_2\equiv U_4$,
$M_6-M_4\equiv M_5-M_3\equiv U_{18}$,
$M_5-M_4\equiv M_6-M_3\equiv U_{20}$,
$M_8-M_4\equiv M_7-M_3\equiv U_{21}$,
$M_7-M_4\equiv M_8-M_3\equiv U_{23}$,
$M_8-M_6\equiv M_7-M_5\equiv U_{15}$
$M_7-M_6\equiv M_8-M_5\equiv U_{16}$.
The subgroup generated by the images of these matrices includes all 4-cycles of $S_4$, hence must be $S_4$. 
In fact, from the data above one can also check 
for any subset $\{L_\infty,L_0,L_{M_1}\}\subseteq \mathcal L'\subseteq\mathcal L$ that
$G_{\mathcal L'}$ is $S_4$ if and only if
$\mathcal L'$ includes two lines $L_{M_i}, L_{M_j}$
with $2<i<j$ such that $M_iM_j\not\equiv I_2$ (note that two distinct elements of order 4 in $S_4$ generate
$S_4$ as long as they do not have trivial product; the ones with $M_iM_j\equiv I_2$ are $(i,j)=(3,4), (5,6), (7,8)$).
This means 
for any subset $\mathcal L'\subset \mathcal L$ containing 
$L_\infty,L_0,L_{M_1},L_{M_3}, L_{M_5}\}$ that $G_{\mathcal L'}=S_4$. In particular, this shows that there are at least
at least 6 projective equivalence classes of finite sets of skew lines with group $S_4$, since here we see examples with 5, 6, 7, 8, 9 and 10 lines, but in fact there are more since $\{L_\infty,L_0,L_{M_1},L_{M_2},L_{M_3},L_{M_5}\}$ and $\{L_\infty,L_0,L_{M_1},L_{M_3},L_{M_5},L_{M_7}\}$ both give $S_4$ but cannot be projectively equivalent. Indeed, $L_\infty,L_0,L_{M_1},L_{M_2}$ all lie on the same smooth quadric, $xw-yz=0$, but no smooth quadric contains any four of the lines $L_\infty,L_0,L_{M_1},L_{M_3},L_{M_5},L_{M_7}$.
\end{example}

\begin{remark}
Each line of $\mathcal L$  is naturally identified with
$\PP^1_{\acfield}$. Namely, $L_\infty$ and $L_0$ are parametrized by
\[
[u:v]\longmapsto [0:0:u:v],
\qquad
[u:v]\longmapsto [u:v:0:0],
\]
respectively, while $L_M$ is parametrized by
\[
[u:v]\longmapsto [u:v:u':v'],
\qquad
\begin{pmatrix}u'\\ v'\end{pmatrix}
=
M\begin{pmatrix}u\\ v\end{pmatrix}.
\]
Under these identifications, an element $g\in{\rm PGL}_2(\acfield)$ acts on
each line through its usual action on $\PP^1_{\acfield}$. Equivalently, choosing any representative $A\in{\rm GL}_2(\acfield)$ of $g$,
the action on the parameter is
\[
[u:v]\longmapsto \left[A\binom{u}{v}\right].
\]
Thus, on the line $L_M$, we have
\[
\left[\binom{u}{v}:M\binom{u}{v}\right]
\longmapsto
\left[A\binom{u}{v}:MA\binom{u}{v}\right]
\]where the notation
$
\left[\binom{x}{y}:\binom{z}{w}\right]
$
means the point $[x:y:z:w]\in \PP^3_{\acfield}$.

This is independent of the chosen representative $A$. In particular,
$G_{\mathcal L}\subset{\rm PGL}_2(\acfield)$ acts on every line of
$\mathcal L$.
\end{remark}

The next observation explains why, once a configuration is normalized to contain
$L_\infty,L_0,L_{I_2}$, conjugating the corresponding subgroup of
${\rm PGL}_2(\acfield)$ does not change the projective equivalence class of the
configuration.

\begin{remark}\label{rem:conjugation-projective}
	An automorphism of $\PP^3_{\acfield}$ preserves $L_\infty$, $L_0$, and $L_{I_2}$
	if and only if it is, up to scalar matrices, a block diagonal matrix
	\[\Psi_A=
	\begin{pmatrix}
		A&0\\
		0&A
	\end{pmatrix}
	\]
	where $A\in {\rm GL}_2(\acfield)$, so $\Psi_A$
	defines a projective automorphism $\psi_A$ of $\PP^3_{\acfield}$ preserving
	$L_\infty$, $L_0$, and $L_{I_2}$, and sends any line $L_M$ skew to $L_\infty$, $L_0$, and $L_{I_2}$ to $L_{AMA^{-1}}$.
	Thus conjugating a subgroup of ${\rm PGL}_2(\acfield)$ corresponds to a
	projective change of coordinates of the normalized configuration.
	Moreover, given a finite set of skew lines $\mathcal L=\{L_\infty, L_0, L_{I_2}, L_{M_4},\ldots, L_{M_s}\}$
	having group $G$, the lines 
	$$\mathcal L'=\{L_\infty, L_0, L_{I_2}, L_{AM_4A^{-1}},\ldots, L_{AM_sA^{-1}}\}$$
	have group $aGa^{-1}\subset {\rm PGL}_2(\acfield)$, where $a$ is the image of $A$ in ${\rm PGL}_2(\acfield)$.
\end{remark}

\begin{remark}\label{3LinesTo3LinesRem}
Given 3 skew lines $\ell_1,\ell_2,\ell_3$ in $\PP^3$,
to map them to $\ell_1'=L_\infty, \ell_2'=L_0, \ell_3'=L_{I_2}$,
pick two planes $H_1,H_2$ containing $\ell_3$. Let $p_{ij}$ be the points $\ell_i\cap H_j$ for $i=1,2$. Let $p_{3j}$ be the point of $\ell_3$ on the line through $p_{1j}$ and $p_{2j}$. Let $H_1',H_2'$ be planes containing $L_{I_2}$ and let $p_{ij}'$ be the points $\ell_i'\cap H_j'$, $i=1,2$. 
Let $p_{3j}'$ be the point of $\ell_3'$ on the line through $p_{1j}'$ and $p_{2j}'$. 
It is convenient to take 
$H_1'$ to be $x-z=0$ and $H_2'$ to be $y-w=0$; then $p_{11}'$ and $p_{12}'$ 
will be $[0:0:0:1]$ and $[0:0:1:0]$,
$p_{21}'$ and $p_{22}'$ 
will be $[0:1:0:0]$ and $[1:0:0:0]$,
and $p_{31}'$ and $p_{32}'$ 
will be $[0:1:0:1]$ and $[1:0:1:0]$.
There clearly are projective
transformations with $p_{ij}'\mapsto p_{ij}$ for $1\leq i,j\leq 2$
since we are mapping 4 linearly general points to 4 linearly general points. With the specific choices for $H_j'$ above, the matrices representing such projective transformations 
have columns where the first column is given (up to scalars) by $p_{22}'$, the second column
by $p_{21}'$, the third by $p_{12}'$ and the fourth by $p_{11}'$. We can then get a matrix that also maps $p_{3j}'\mapsto p_{3j}$ for $1\leq j\leq 2$ by multiplying each column by an appropriate scalar; this is because 
$p_{1j}',p_{2j}',p_{3j}'$ are collinear and so are 
$p_{1j},p_{2j},p_{3j}$. The inverse of the matrix
then defines a transformation taking $\ell_i$ to $\ell_i'$ for each $i$.
\end{remark}

\section{A finiteness result on the lines associated to a given group}

By \cite[Proposition~4.1]{Beauville2010}, finite subgroups of ${\rm PGL}_2(\acfield)$ of 
the same isomorphism type and of order prime to the characteristic are conjugate.
Thus by Remark \ref{rem:conjugation-projective}, such a conjugation is realized, in the normalized matrix model, by a projective change of coordinates. Hence, to study which sets of skew lines have group isomorphic to a given group $G$,
it is enough, up to projective equivalence, to fix one embedding $G\subset {\rm PGL}_2(\acfield).$

For this fixed subgroup $G$, Theorem \ref{mainThm} shows that there are only
finitely many normalized configurations
$\mathcal L=\{L_\infty,L_0,L_{I_2},\ell_4,\ldots,\ell_s\}$
with $G_{\mathcal L}= G$.

\begin{remark}\label{rem:representatives-matter}
A point which will be important in the proof is that a line $L_M$ depends on
the matrix $M\in {\rm GL}_2(\acfield)$, not only on its class
$[M]\in {\rm PGL}_2(\acfield)$. Thus, even after fixing a finite subgroup
$G\subset {\rm PGL}_2(\acfield)$, an element $g\in G$ represented by a matrix $M$ is also represented by $cM$ for every nonzero scalar $c$, but the corresponding lines $L_{cM}$ are by Lemma \ref{MatrixLineAssoc}
distinct for distinct values of $c$.
For instance, $I_2$ and $cI_2$ have the same image in ${\rm PGL}_2(\acfield)$,
but the lines
\[
L_{I_2}\colon z=x,\ w=y
\qquad\text{and}\qquad
L_{cI_2}\colon z=cx,\ w=cy
\]
are distinct if $c\neq 1$. Therefore, fixing the subgroup
$G\subset {\rm PGL}_2(\acfield)$ does not by itself give only $|G|+2$ possible
lines. In other words, Theorem \ref{mainThm} is not merely a count of the elements of $G$; it is a bound on the possible representatives of those elements which are
compatible with the geometry of the skew-line configuration.
\end{remark}

\begin{theorem}\label{mainThm}
Let $G$ be a finite nontrivial subgroup of ${\rm PGL}_2(\acfield)\cong{\rm Aut}(L_\infty)$. 
Then there is a unique minimal finite set 
$\calm_G$ of lines such that 
whenever
	$\mathcal L=\{L_\infty,L_0,L_{I_2},\ell_4,\ldots,\ell_s\}
	$
	is a set of skew lines with $G_{\mathcal L}= G$, then
	$
	\mathcal L\subseteq \calm_G.$
Moreover, 
$$|\calm_G|\le (|G|-2)(|G|-1)^2+2.$$
\end{theorem}

\begin{proof} 
There is clearly a unique minimal set $\calm_G$; it is the union 
of all sets of the form 
$$\mathcal L=\{L_\infty,L_0,L_{I_2},\ell_4,\ldots,\ell_s\}$$ 
such that $G_{\mathcal L}=G$.
The main point is to show that this union is finite, and to give the stated bound.
So assume $G=G_{\mathcal L}$ is a nontrivial finite group for some finite set 
$\mathcal L=\{L_\infty,L_0,L_{I_2},\ell_4,\ldots,\ell_s\}$ of skew lines.

Note that ${\rm GL}_2(\acfield)$ is the complement of a hypersurface in the vector space 
${\rm Mat}_{2\times 2}(\acfield)\cong\acfield^4$ of 2 by 2 matrices. Modding out by scalars
gives a map defined away from the zero matrix in ${\rm Mat}_{2\times 2}(\acfield)$
to $\PP^3_{\acfield}$ which maps
${\rm GL}_2(\acfield)$ to ${\rm PGL}_2(\acfield)$ as an open subset of $\PP^3_{\acfield}$.
For each $g\in G$, pick a representative $M_g\in {\rm GL}_2(\acfield)$ that maps to $g$.
(We do not require that $\det M_g=1$.  Note that $G$ is defined over a finite extension of the prime
field. To do computations, it is convenient to pick matrices $M_g$ defined over the same field as $G$.)
If $g,h\in G$ with $g\neq h$, then $M_g$ and $M_h$ are linearly independent (since otherwise
$cM_g=M_h$ for some nonzero scalar $c$, which implies $M_g$ and $M_h$ map to the same
element $g=h$ of $G$). For nonzero matrices $M,N\in{\rm Mat}_{2\times 2}(\acfield)$,
we will write $M\equiv N$ if they map to the same element of $\PP^3_{\acfield}$.

By Lemma \ref{MatrixLineAssoc} we have matrices $M_i\in{\rm GL}_2(\acfield)$ such that 
$\ell_i=L_{M_i}$ for $i\geq 4$, and we set $M_3=I_2$.
Then $G$ is generated by the images in ${\rm PGL}_2(\acfield)$ of
$M_3,\ldots,M_s$ together with $M_j-M_i$ for $3\leq i<j\leq s$.
In particular, for each matrix $M_i$ there is a $g_i\in G$ and a nonzero $c_i\in\acfield$ such that $M_i= c_iM_{g_i}$.

Now let $g,h\in G$ be such that no two of $1,g,h$ are equal and 
consider the equation $c_{g,h}M_g-I_2=c_hM_h$ (or equivalently $c_{g,h}M_g-c_hM_h=I_2$). Since $M_g$ and $M_h$
are linearly independent there is at most one solution $(c_{g,h},c_h)\in\acfield^2$.
Thus there are at most $(|G|-1)(|G|-2)$ pairs $(c_{g,h},c_h)\in\acfield^2$ such that $c_{g,h}M_g-I_2=c_hM_h$.

In particular, say $i>3$.  
Since $M_i-M_3$ maps to an element of $G$, we must have $M_i-M_3=c_iM_{g_i}-I_2\equiv M_g$
for some $g\in G$ and hence $c_iM_{g_i}-I_2=c_gM_g$ for some nonzero $c_g\in\acfield$. 
Thus we cannot have $g=1$ or $g=g_i$: if $g=1$, then $c_iM_{g_i}=(c_g+1)I_2$,
so $M_{g_i}\equiv I_2$ hence $g_i=1$ and so $i=3$ contrary to assumption, while if $g=g_i$, then 
$(c_i-c_g)M_{g_i}=I_2$ so again $g_i=1$. 
This means that $(c_i,c_g)$ is among the $(|G|-1)(|G|-2)$ pairs $(c_{g,h},c_h)\in\acfield^2$
identified above, so each matrix $M_i$ not congruent to $I_2$ comes from the set
of matrices of the form $c_{g,h}M_g$, and there are at most $(|G|-1)(|G|-2)$ matrices $c_{g,h}M_g$.

There remains to bound the number of matrices $M_i\equiv I_2$. For any such $M_i$ we have $M_i=c_iI_2$ for some $c_i\neq0$. For this 
let $c_{g,h}M_g$ be one of the $(|G|-1)(|G|-2)$ matrices as above with $g\neq1$ and consider
the equation $c_{g,h}M_g-c_1I_2=c_eM_e$ (or equivalently $c_eM_e+c_1I_2=c_{g,h}M_g$). 
There is no solution with $c_1\neq0$ when $e=g$ or $e=1$ (since $g\neq1$ implies $M_g$ and $I_2$ are linearly independent),
and for each $e$ with $1\neq e\neq g$, there is at most one solution $(c_e,c_1)$ and any such solution has $c_1\neq0$.
Thus altogether there are at most
$(|G|-1)(|G|-2)^2$ triples $(c_e,c_1,c_{g,h})$ with $c_1\neq0$ (since there are at most $(|G|-1)(|G|-2)$ matrices $c_{g,h}M_g$
and for each at most $|G|-2$ matrices $M_e$ for which there is a triple $(c_e,c_1,c_{g,h})$ with $c_1\neq0$).
To relate this to bounding the number of matrices $M_i\equiv I_2$ that can arise, note that
$G$ being nontrivial ensures there is some matrix $M_j=c_{g,h}M_g$ not congruent to $I_2$ among the
matrices $M_j$ defining the lines $L_{M_3},\ldots,L_{M_s}$. Moreover, given 
$M_i=c_1I_2$ for some $c_1\neq0$, the difference $M_j-M_i=c_{g,h}M_g-c_1I_2$ maps to an element
of $G$, hence $c_{g,h}M_g-c_1I_2=c_eM_e$ for some $e\in G$ and some scalar $c_e\neq0$,
hence $e\neq g, 1$. 
Thus $(c_e,c_1,c_{g,h})$ is among the $(|G|-1)(|G|-2)^2$ triples identified above.
In particular, there are at most $(|G|-1)(|G|-2)^2$ matrices $M$ congruent to $I_2$ that can arise 
as an $L_M$ in the set $\mathcal L$.

Thus altogether the matrices $M_3,\ldots,M_s$ are always elements of a fixed set
of at most $(|G|-1)(|G|-2)+(|G|-1)(|G|-2)^2=(|G|-1)^2(|G|-2)$ matrices.
Accounting for $L_\infty$ and $L_0$ gives the bound of $(|G|-1)^2(|G|-2)+2$.
\end{proof}

\begin{example}The bound $(|G|-1)^2(|G|-2)+2$ is: 
1212 for $G=A_4$;    11640 for $G=S_4$;
  201900 for $G=A_5$.
\end{example}

\begin{corollary}\label{c.  infinite or cyclic.}
Assume ${\rm char}(\acfield)=0$ and 
consider a set ${\mathcal L}=\{L_1,\ldots, L_s\}$ of distinct skew lines in $\PP^3_{\acfield}$.
If $s>201900$, then $G_{\mathcal L}$ is either an infinite group or finite cyclic.
\end{corollary}

\begin{proof}
By \cite[Theorem C]{Faber}, if $G$ is a finite subgroup of ${\rm PGL}_2(\acfield)$ of order
prime to ${\rm char}(\acfield)$ (which always holds in characteristic 0), then $G_{\mathcal L}$ is cyclic, dihedral, or isomorphic to one of $A_4$, $S_4$, or $A_5$. However \cite{Favacchio} shows $G_{\mathcal L}$
is never dihedral. But for $G_{\mathcal L}$ to be either $A_4$, $S_4$ or $A_5$, by Theorem \ref{mainThm}
${\mathcal L}$ cannot have more than $|\calm_{G_{\mathcal L}}|$ lines. Since for 
$A_4$, $S_4$ or $A_5$ we always have $|\calm_{G_{\mathcal L}}|\leq 201\,900$, it follows that
$G_{\mathcal L}$ is either an infinite group or finite cyclic.
\end{proof}

This corollary suggests the possibility of classifying all sets of skew lines $\mathcal L$ in $\PP^3_{\mathbb C}$ for which 
$G_{\mathcal L}$ is finite and nonabelian. The fact that the bound on $|\calm_G|$ given by Theorem \ref{mainThm}
is rather large may make the prospect intimidating, but the bound is not very tight. In the next section we 
prove this by carrying out the classification for $A_4$ in characteristic $0,$ from which we see that $|\calm_{A_4}|=|A_4|-1=11$.

\begin{definition}
Given a finite subgroup $G\subseteq {\rm PGL}_2(\acfield)$, let 
$$\mathcal A_G=\{a\in {\rm PGL}_2(\acfield): aGa^{-1}=G\}.$$
\end{definition}

\begin{remark}
Note that ${\rm PGL}_2(\acfield)$ acts on ${\rm GL}_2(\acfield)$ by conjugation: for any elements $m\in {\rm PGL}_2(\acfield)$ and
$B\in {\rm GL}_2(\acfield)$, $MBM^{-1}$ is the same for all representatives $M\in {\rm GL}_2(\acfield)$ for $m$, so $mBm^{-1}$
is well defined. Likewise, ${\rm GL}_2(\acfield)$ acts by conjugation on ${\rm PGL}_2(\acfield)$.
\end{remark}

The following proposition exhibits some immediate properties.

\begin{proposition}\label{A_GProp}
Let $G$ be a subgroup of ${\rm PGL}_2(\acfield)$. We have the following.
\begin{itemize}
\item[(1)] $\mathcal A_G$ is a subgroup of ${\rm PGL}_2(\acfield)$ and $G$ is a subgroup of $\mathcal A_G$. 
\item[(2)] If $b\in{\rm PGL}_2(\acfield)$, then $\mathcal A_{bGb^{-1}}=b\mathcal A_Gb^{-1}$.
\item[(3)] $\mathcal A_G$ acts by conjugation on $\calm_G$.
\item[(4)] Assume $\acfield=\mathbb C$ with $\mathcal L=\{L_\infty,L_0, L_{I_2}, L_{M_4},\ldots, L_{M_s}\}$ being skew lines 
having group $G\subset {\rm PGL}_2(\mathbb C)$ where $G$ is finite. 
Then for each choice of $L_a,L_b,L_c\in\mathcal L$, there is a $\Lambda\in{\rm PGL}_4(\mathbb C)$
with $\Lambda(L_a)= L_\infty, \Lambda(L_b)= L_0, \Lambda(L_c)=L_{I_2}$, and for any such $\Lambda$,
there is a $B\in{\rm PGL}_2(\mathbb C)$ such that $\psi_B\Lambda\mathcal L$
has group $G$, where $\psi_B$ is as in 
Remark \ref{rem:conjugation-projective}. Moreover, if $\Lambda'\in{\rm PGL}_4(\mathbb C)$ also has 
$\Lambda'(L_a)= L_\infty, \Lambda'(L_b)= L_0, \Lambda'(L_c)=L_{I_2}$ where $\Lambda'\mathcal L$ has group $G$, then there is a matrix $A\in{\rm GL}_2(\mathbb C)$
mapping to an element of $\mathcal A_G$ such that $\Lambda'\equiv\psi_A\psi_B\Lambda$.
\end{itemize}
\end{proposition}

\begin{proof}
Items (1,2) are routine and left to the reader.

(3) Let $M\in \calm_G$. Then there is a finite set $\mathcal L$ of skew lines including 
$L_\infty, L_0, L_{I_2}$ and $L_M$ whose group is $G$. By Remark \ref{rem:conjugation-projective}, 
for each element $\alpha$ of $\mathcal A_G$
there is a finite set of skew lines $\mathcal L'$ including $L_\infty, L_0, L_{I_2}$ and $L_{\alpha M\alpha^{-1}}$ 
whose group is $\alpha G\alpha^{-1}=G$. 
Thus $\alpha M\alpha^{-1}\in \calm_G$.

(4) By Remark \ref{3LinesTo3LinesRem}, there is a $\Lambda$ with 
$$\Lambda(L_a)= L_\infty, \Lambda(L_b)= L_0, \Lambda(L_c)=L_{I_2},$$
and $\Lambda\mathcal L=\{L_\infty,L_0, L_{I_2},L_{M_4'},\ldots,L_{M_s'}\}$
for some matrices $M_i'\in{\rm GL}_2(\acfield)$.
The images of $I_2,M_4',\ldots,M_s'$ in ${\rm PGL}_2(\acfield)$ and their differences 
generate a group $G'$; by Remark \ref{IndependentRem}, $G'$ is isomorphic to $G$. 
But over $\mathbb C$, all instances of $G$ in ${\rm PGL}_2(\mathbb C)$
are conjugate, so there is a $b\in {\rm PGL}_2(\mathbb C)$ such that
$bG'b^{-1}=G$. Taking a $B\in{\rm GL}_2(\mathbb C)$
that maps to $b$, we get $\psi_B\Lambda$ as desired.
Finally, $a=\Lambda'\Lambda^{-1}\psi_B^{-1}$ fixes the lines $L_\infty, L_0, L_{I_2}$ (set-wise, but not 
necessarily point-wise) so $a=\psi_A$ for some matrix $A\in {\rm GL}_2(\mathbb C)$.
Thus the group associated to $a\mathcal L$ is $AGA^{-1}$. but both $\psi_B\Lambda$ and $\Lambda'$
preserve $G$ as a subset of ${\rm PGL}_2(\acfield)$, so $AGA^{-1}=G$, hence $A$ maps to an element of 
$\mathcal A_G$ in ${\rm PGL}_2(\acfield)$.
\end{proof}

\begin{example}\label{A4-S4Examples} As examples,
we show for $G\cong S_4$ and $\acfield=\mathbb C$ that we have $\mathcal A_G=G$, and for $G\cong A_4$ and $\acfield=\mathbb C$, that there is a unique $H\subset {\rm PGL}_2(\acfield)$ 
with $H\cong S_4$ containing $G$ and that we have $\mathcal A_G=H$.

First consider $G\cong S_4$. By \cite{Faber}, all instances of $S_4$ in ${\rm PGL}_2(\acfield)$ are conjugate, so by 
Proposition \ref{A_GProp}(2), it is enough to confirm
$\mathcal A_G=G$ for the instance given in Example \ref{S4Example}. If $\alpha\in \mathcal A_G$, then $\alpha$
induces a permutation on the elements of $G$, so we have a homomorphism $\mathcal A_G\to {\rm Perms}(G)$.
We now show this is injective and hence $\mathcal A_G$ is finite. 
Suppose $b$ is in the kernel and let $M\in {\rm GL}_2(\acfield)$ represent an element of $G$.
Then $bMb^{-1}=cM$ for some scalar $c$. Since conjugation preserves determinants, we have $c^2=1$ so
$bMb^{-1}=\pm M$. The only invertible matrices $B$ with $BU_2B^{-1}=\pm U_2$ and $BU_5B^{-1}=\pm U_5$
are scalar matrices, hence $b$ is trivial.

Moreover, $\mathcal A_G$ is a nonabelian subgroup of ${\rm PGL}_2(\acfield)$
(since by Proposition~\ref{A_GProp}(1) it contains $G\cong S_4$)
and by \cite{Faber}, over $\mathbb C$, a finite nonabelian subgroup of ${\rm PGL}_2(\acfield)$ must be either
a dihedral group, $A_4$, $S_4$ or $A_5$. But neither $A_4$ nor $A_5$ contain a subgroup isomorphic to $S_4$,
and a dihedral group has an abelian subgroup of index 2, so every subgroup of a dihedral group 
is either abelian or has an abelian subgroup of index 2. Thus $G$ cannot be a subgroup of a dihedral group.
Thus $\mathcal A_G=G$.

Now consider $G\cong A_4$. We can, mimicking what we did before, assume $G$ is the $A_4$ given in Example \ref{S4Example}. 
Let $H$ be the $S_4$ given in the example, so $G$ is a normal subgroup of $H$.
Thus $H\subseteq \mathcal A_G$. Now the argument given before shows that $\mathcal A_G=H$.
If $G$ is contained in two $S_4$'s, say $H$ and $bHb^{-1}$, then 
$bGb^{-1}=G$ so $b\in \mathcal A_G=H$ so $H=bHb^{-1}$, hence the $S_4$ containing $G$ is unique.
\end{example}

\section{The classification of skew lines with group $A_4$}
It is convenient to work with a particular representation of $A_4$ in
${\rm PGL}_2(\acfield)$. Since all such subgroups are conjugate, this choice
does not affect the classification up to projective equivalence. More
precisely, we choose the twelve matrices $U_i\in{\rm GL}_2(\acfield)$ given in 
Example \ref{S4Example}.
Their 
images in ${\rm PGL}_2(\acfield)$ form a subgroup isomorphic to $A_4$.
\medskip
\paragraph{\bf Notation.} We use $U$ for both the matrix in ${\rm GL_2(K)}$ and its 
image $[U]$ in ${\rm PGL_2(K)}$. So 
we suppress the notation for the image of a matrix in
${\rm PGL}_2(\acfield)$.
In particular, we will use $U_i-U_j$ for the image $[U_i-U_j]$ of the matrix $U_i-U_j$, and when a matrix $U$ is said to belong to $A_4$, or
to have a given order, this refers to its image in
${\rm PGL}_2(\acfield)$. 

Lemma \ref{lem:A4-possible-lines} finds a finite set $\calm_{A_4}'$ which contains $\calm_{A_4}$.
In fact, it turns out that $\calm_{A_4}'=\calm_{A_4}$, but to confirm this we will need to show that
every element of $\calm_{A_4}'$ occurs in some set $\mathcal L$ containing $L_\infty, L_0$ and $L_{I_2}$
with $G_{\mathcal L}=A_4$. We do this in Theorem \ref{thm:A4-classification}.

\begin{lemma}\label{lem:A4-possible-lines}
Assume that $\operatorname{char}(K)=0$, and fix the subgroup
$G\subset {\rm PGL}_2(\acfield)$ isomorphic to $A_4$ represented by the
matrices 
$
U_1,U_2,\ldots,U_{12}
$ 
mentioned above.

Let
$$
{\mathcal L}=\{L_\infty,L_0,L_{M_1},\ldots,L_{M_r}\},\qquad M_1=I_2,
$$
be a set of skew lines such that $G_{\mathcal L}=G$. Then every matrix $M_i$ with
$i\geq 2$ is equal to one of the matrices 
$
U_5,\ldots,U_{12}.
$ 
Consequently,
$$
{\mathcal L}\subseteq
\calm_{A_4}
\subseteq
\calm_{A_4}'=\{L_\infty,L_0,L_{I_2},L_{U_5},\ldots,L_{U_{12}}\},
$$
and hence 
$
|\calm_{A_4}'|=11.
$
\end{lemma}

\begin{proof}
Let $M=M_i$ with $i\geq 2$. Since $G_{\mathcal L}=G$, the image of $M$ in
${\rm PGL}_2(\acfield)$ belongs to $G$. If $M$ is not scalar, then
$$
M=cU
$$
for some nonzero scalar $c$ and some $U=U_j$ with $j\in\{2,\ldots,12\}$.
We first treat the non-scalar case assuming that $U$ has order $2$. Then $U$ has eigenvalues
$i$ and $-i$. Since $M-I_2=cU-I_2$ is a polynomial in $U$, its image commutes
with $U$ in $A_4$. The centralizer in $A_4$ of an element of order $2$ is the
Klein four subgroup, so the image of $M-I_2$ is either the identity or has
order $2$.

The image of $M-I_2$ cannot be the identity, since then $M-I_2$ would be
scalar and hence $M$ itself would be scalar, contrary to the assumption that
$U$ has order $2$. Therefore $M-I_2$ must have order $2$. The eigenvalues of
$cU-I_2$ are $ci-1$ and $-ci-1$, so this forces
$$
(ci-1)^2=(-ci-1)^2.
$$
This gives $c=0$, a contradiction.
Hence no scalar multiple of a representative
of an element of order $2$ in $G$ can occur.

Now suppose that $U$ has order $3$. 
Then $U$ has eigenvalues
$t$ and $t^*$, where
$$
t^2-t+1=0,\qquad t+t^*=1,\qquad tt^*=1.
$$
Again $M-I_2=cU-I_2$ is a polynomial in $U$, so its image commutes with $U$.
The centralizer in $A_4$ of an element of order $3$ is the cyclic subgroup
generated by that element. Since the image of $M-I_2$ cannot be the identity,
it must also have order $3$. Therefore
$$
(ct-1)^3=(ct^*-1)^3.
$$
Using $t+t^*=1$ and $tt^*=1$, this simplifies to 
$
c-1=0.
$
Hence $c=1$.

Conversely, for such a representative $U$ we indeed have $U-I_2$ projecting
to $G$. For our chosen representatives of the order-$3$ elements, the eigenvalues
are $t$ and $t^*$, so the trace is $1$ and the determinant is $1$. 
Hence
$$
U^2-U+I_2=0,
\qquad
\text{and therefore}
\qquad
I_2=U(I_2-U).
$$
Thus $-U+I_2$ represents the inverse of $U$ in ${\rm PGL}_2(\acfield)$.

It remains to consider the case where $M$ is scalar. Write $M=cI_2$. 
Since $G_{\mathcal L}=G\simeq A_4$, some non-scalar matrix must occur among the matrices
defining the lines. By the non-scalar case just proved, such a matrix must be
one of $U_5,\ldots,U_{12}$, hence represents an element of order $3$. Since $U-cI_2$ must
project to an element of $G$, equivalently
$$
c^{-1}U-I_2
$$
must project to an element of $G$. By the preceding paragraph this forces
$c^{-1}=1$, hence $c=1$. Thus the only scalar matrix which can occur is
$I_2$, giving the already fixed line $L_{I_2}$.

The description of $\calm_{A_4}'$ and the equality $|\calm_{A_4}'|=11$ follow
immediately, since $A_4$ has eight elements of order $3$.
\end{proof}

In the next lemma we describe in more detail which subsets of $\calm_{A_4}'$ can give rise to the group~$A_4$.

\begin{lemma}\label{lem:A4-admissible-pairs}
Let $U,V\in\{U_5,\ldots,U_{12}\}$. Thus $U$ and $V$ have order $3$ in
$A_4$, in the sense of our notational convention.

If the two lines $L_U$ and $L_V$  occur simultaneously in a set of skew lines
with associated group contained in the fixed subgroup $A_4$ then $UV$ has
order $2$.

Conversely, if $UV$ has order $2$, then $U-V$ is
invertible and projects to an element of $A_4$. Hence
$$
\{L_\infty,L_0,L_{I_2},L_U,L_V\}
$$
is a set of five skew lines with associated group $A_4$.

There are exactly twelve unordered pairs $\{U,V\}$ with this property.
The union of these 12 pairs is
$\{U_5,\ldots,U_{12}\}$. Moreover, there is no triple $U,V,W\in\{U_5,\ldots,U_{12}\}$ for which all three
products
$$
UV,\quad
UW,\quad
VW
$$
have order $2$. 
\end{lemma}

\begin{proof}
The condition that $L_U$ and $L_V$ occur simultaneously implies that
$U-V$ is invertible and that its image in ${\rm PGL}_2(\acfield)$ belongs to $A_4$. 
A direct calculation with the representatives $U_5,\ldots,U_{12}$ shows
that, for $i<j$, the matrix $U_i-U_j$ is invertible and projects to an
element of $A_4$ if and only if $U_iU_j$ has order $2$ in $A_4$.
Equivalently, this can be read off from the multiplication table of the
subgroup of ${\rm SL}_2(\acfield)$ generated by the displayed matrices
$U_1,\ldots,U_{12}$ and their negatives. This subgroup maps two-to-one onto
our fixed copy of $A_4\subset{\rm PGL}_2(\acfield)$; it is the binary
tetrahedral group.

We enclose also the compatibility graph. Two vertices (matrices $U_i,U_j$) are 
joined by an edge if their difference $U_i-U_j$ maps to a group element.
Note that every $U_i$ with $5\leq i\leq 12$ occurs as an endpoint of an edge.

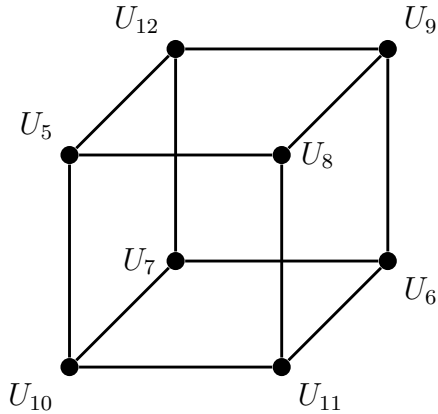
\begin{figure}[ht]
\centering
\begin{tikzpicture}[
    scale=1.4,
    every node/.style={font=\Large},
    vertex/.style={circle, fill=black, inner sep=2.4pt}
  ]

\node[vertex,label=above left:$U_5$]    (U5)  at (0,2) {};
\node[vertex,label=right:$U_8$]   (U8)  at (2,2) {};
\node[vertex,label=below right:$U_{11}$](U11) at (2,0) {};
\node[vertex,label=below left:$U_{10}$] (U10) at (0,0) {};

\node[vertex,label=above left:$U_{12}$] (U12) at (1,3) {};
\node[vertex,label=above right:$U_9$]   (U9)  at (3,3) {};
\node[vertex,label=below right:$U_6$]   (U6)  at (3,1) {};
\node[vertex,label=left:$U_7$]    (U7)  at (1,1) {};

\draw[line width=.015in] (U5) -- (U8);
\draw[line width=.015in] (U8) -- (U11);
\draw[line width=.015in] (U11) -- (U10);
\draw[line width=.015in] (U10) -- (U5);

\draw[line width=.015in] (U12) -- (U9);
\draw[line width=.015in] (U9) -- (U6);
\draw[line width=.015in] (U6) -- (U7);
\draw[line width=.015in] (U7) -- (U12);

\draw[line width=.015in] (U5)  -- (U12);
\draw[line width=.015in] (U8)  -- (U9);
\draw[line width=.015in] (U11) -- (U6);
\draw[line width=.015in] (U10) -- (U7);

\end{tikzpicture}

\caption{\label{A4CompGraph}
Compatibility graph for the $A_4$ case. The vertices are
$U_5,\dots,U_{12}$. Two vertices $U_i$ and $U_j$ are joined by an edge
if and only if $[U_i-U_j]\in A_4$.}
\end{figure}

It remains only to count the relevant pairs inside $A_4$. To this end, we use the standard realization of $A_4$ as the group of even permutations of four
letters $a,b,c,d$. 
The elements of order $3$ are the eight $3$-cycles, while the elements of
order $2$ are the three double transpositions, i.e., permutations of the form
$(ab)(cd)$.
If $U$ corresponds to 
$
(abc),
$
then the three elements $V$ of order $3$ for which $UV$ has order $2$ are
precisely the $3$-cycles involving the fourth letter $d$ and two of
$a,b,c$. For example, with the convention that permutations are composed from right to left,
$$
(abc)(abd)=(ac)(bd).
$$
Thus each of the eight $3$-cycles is paired with exactly three others.
Hence the number of unordered admissible pairs is
$
\dfrac{8\cdot 3}{2}=12.
$

It remains to see that no three elements can be pairwise admissible.
By symmetry it is enough to start with the cycle $(abc)$. The three
$3$-cycles whose product with $(abc)$ is a double transposition are
$$
(abd),\qquad (adc),\qquad (bcd).
$$
Indeed,
$$
(abc)(abd)=(ac)(bd),\quad
(abc)(adc)=(ad)(bc),\quad
(abc)(bcd)=(ab)(cd).
$$
But the product of any two of these three cycles is again a $3$-cycle, not
a double transposition. Hence no admissible triple exists.
\end{proof}
We summarize our considerations in the following theorem, which additionally shows that up to projective equivalence there is a unique set $L$ of skew lines such that $G_{\mathcal L}\simeq A_4$.

\begin{theorem}\label{thm:A4-classification}
Assume that $\operatorname{char}(K)=0$. Up to projective equivalence, there
is a unique set $\mathcal L$ of skew lines in $\PP^3_\field$ such that
$$
G_{\mathcal L}\simeq A_4.
$$
This set consists of five lines. More precisely, after a projective change
of coordinates it is of the form
$$
{\mathcal L}=\{L_\infty,L_0,L_{I_2},L_U,L_V\},
$$
where $U,V\in\{U_5,\ldots,U_{12}\}$ and $UV$ has order $2$ in $A_4$. It follows that $\calm_{A_4}=\calm_{A_4}'$.
\end{theorem}

\begin{proof}
Let ${\mathcal L}$ be a set of skew lines with $G_{\mathcal L}\simeq A_4$. By conjugating the
corresponding subgroup of ${\rm PGL}_2(\acfield)$ and, using Remark~\ref{rem:conjugation-projective}, we
may assume that $G_{\mathcal L}$ is the fixed subgroup $G\simeq A_4$ represented by
$U_1,\ldots,U_{12}$. Since any three skew lines are projectively equivalent,
we may also assume that
$
L_\infty, L_0, L_{I_2}
$
belong to~${\mathcal L}$.

The set ${\mathcal L}$ cannot have only these three lines, since then $G_{\mathcal L}$ is trivial.
It also cannot have only one further line $L_M$, since then $G_{\mathcal L}$ is generated
by the images of $M$ and $M-I_2$, which commute; hence $G_{\mathcal L}$ would be abelian, see also \cite[Theorem 2.1.22]{POLITUS3}. Hence ${\mathcal L}$ has at least five lines.

By Lemma~\ref{lem:A4-possible-lines}, every additional line is of the form
$L_U$, where $U\in\{U_5,\ldots,U_{12}\}$ has order~$3$ in $A_4$. By
Lemma~\ref{lem:A4-admissible-pairs}, two such lines can occur together only
when the product of the corresponding order-$3$ elements has order $2$, and
no three additional such lines can occur simultaneously. Thus ${\mathcal L}$ has exactly
five lines and is of the form
$$
\{L_\infty,L_0,L_{I_2},L_U,L_V\},
$$
with $U,V$ of order $3$ and $UV$ of order $2$ in $A_4$.

Conversely, Lemma~\ref{lem:A4-admissible-pairs} shows that every such
five-line set has associated group $A_4$, which confirms that $\calm_{A_4}=\calm_{A_4}'$.

It remains only to note that all such five-line sets are projectively
equivalent. By Lemma~\ref{lem:A4-admissible-pairs}, the admissible pairs are
the twelve unordered pairs of $3$-cycles whose product has order $2$. The
conjugation action of $A_4$ on these pairs is transitive. Therefore, by 
Remark~\ref{rem:conjugation-projective}, conjugating one admissible pair to another is induced by a
projective automorphism preserving
$
L_\infty,\ L_0,\ L_{I_2}.
$ 
Hence all admissible five-line configurations are projectively equivalent.
\end{proof}

\begin{remark}
The theorem shows that a configuration with associated group $A_4$ consists
of five lines. This does not contradict the equality $|\calm_{A_4}|=11$.
The set $\calm_{A_4}$ is the union of all lines which can occur in a
configuration $\mathcal L$ containing
$L_\infty,\ L_0,\ L_{I_2}$
such that $G_{\mathcal L}=\{[U_1],\ldots,[U_{12}]\}\cong A_4$.  
Each particular such configuration $\mathcal L$ contains only two of the eight order-$3$ lines $L_{U_i}$, $5\leq i\leq 12$,
appearing in $\calm_{A_4}$.
\end{remark}

\begin{corollary}\label{Cor:AutForA4}
Let $\mathcal L$ be a finite set of skew lines in $\PP^3_{\mathbb C}$ 
with group $A_4$ (thus $\mathcal L$ consists of an appropriate set of 5 lines). Then the group ${\rm Aut}(\mathcal L)\subset{\rm Aut}(\PP^3_{\mathbb C})$
of projective transformations that permute the lines in $\mathcal L$ is isomorphic to $S_5$.
\end{corollary}

\begin{proof} As  noted in Remark \ref{IndependentRem}, up to isomorphism, the group ${\rm Aut}(\mathcal L)$ is 
independent of the choice of coordinates on $\PP^3_{\mathbb C}$,
so by Theorem \ref{thm:A4-classification} we may choose $\mathcal L=\{L_\infty, L_0, L_{I_2},L_{U_5},L_{U_8}\}$.

We have a homomorphism $\pi:{\rm Aut}(\mathcal L)\to S_5$, since each automorphism permutes the lines in $\mathcal L$.
To show $\pi$ is injective suppose $\Lambda\in {\rm GL}_4(\mathbb C)$ represents
an element of ${\rm PGL}_4(\mathbb C)$
that preserves all five lines. Since $\Lambda$ preserves $L_\infty, L_0$ and $L_{I_2}$, we can represent $\Lambda$ as a block matrix 
$\Lambda=\begin{pmatrix}
D & 0\\
0 & D\\
\end{pmatrix}$, where $D\in {\rm GL}_2(\mathbb C)$. Points of $L_{U_5}$
are of the form
$\begin{pmatrix}
x\\
y\\
U_5\begin{pmatrix}
x\\
y\\
\end{pmatrix}
\end{pmatrix}$, which under $\Lambda$ map to
$\begin{pmatrix}
D\begin{pmatrix}
x\\
y\\
\end{pmatrix}\\
DU_5\begin{pmatrix}
x\\
y\\
\end{pmatrix}\\
\end{pmatrix}$.
In order for $L_{U_5}$ to be mapped to itself (and likewise for $L_{U_8}$), we thus need $DU_5=U_5D$ and $DU_8=U_8D$. For 
$D=\begin{pmatrix}
a & b\\
c & d\\
\end{pmatrix}$, this means
$c=0$, $d=a+(-2t+1)b$ and $b=0$ so
up to scalars 
$D\equiv I_2$, hence $\pi$ is injective.
To show $\pi$ is surjective, it suffices to exhibit matrices $\Lambda$ that swap consecutive lines in the list 
$L_\infty, L_0, L_{I_2},L_{U_5},L_{U_8}$
and preserve the other three.
Define
$A=\begin{pmatrix}
2t-1 & 1\\
1 & 1-2t
\end{pmatrix}$,
$B=\begin{pmatrix}
1 & 1-2t\\
1-2t & -1
\end{pmatrix}$, and
$C=\begin{pmatrix}
1 & 1\\
1-2t & -1
\end{pmatrix}$.
Then: \\
$\begin{pmatrix}
0 & B\\
B & 0
\end{pmatrix}$
swaps $L_\infty$ with $L_0$
and maps each of the other lines to themselves;
$\begin{pmatrix}
B & 0\\
B & -B
\end{pmatrix}$
swaps $L_0$ with $L_{I_2}$
and maps each of the other lines to themselves;
$\begin{pmatrix}
C & 0\\
0 & U_5C
\end{pmatrix}$
swaps $L_{I_2}$ with $L_{U_5}$
and maps each of the other lines to themselves; and
$\begin{pmatrix}
A & 0\\
0 & A
\end{pmatrix}$
swaps $L_{U_5}$ with $L_{U_8}$
and maps each of the other lines to themselves.
Thus $\pi$ is surjective.
\end{proof}

\begin{remark}\label{rem:GroupoidForS_4}
In the situation of Corollary \ref{Cor:AutForA4}
there is a nice connection between ${\rm Aut}(\mathcal L)$ and the groupoid ${\mathcal G}_{\mathcal L}$ of $\mathcal L$ (see Remark \ref{IndependentRem} and \cite{POLITUS3}). The objects of ${\mathcal G}_{\mathcal L}$ are the lines $L\in\mathcal L$. For each choice of two lines $L,L'\in\mathcal L$, the Hom set ${\rm Hom}_{{\mathcal G}_{\mathcal L}}(L,L')$ is a set of isomorphisms $L\to L'$ defined by the geometry of the lines in $\mathcal L$.  
But in the context of 
five lines $\mathcal L=\{\ell_1,\ldots,\ell_5\}$ with
$G_{\mathcal L}\cong A_4$ we have, by direct calculation in the case of 
$\mathcal L=\{L_\infty,L_0,L_{I_2},L_{U_5},L_{U_8}\}$ using the generators for ${\rm Aut}(\mathcal L)$ as given in the proof of Corollary \ref{Cor:AutForA4} and the interpretation of the groupoid given in \cite{Favacchio},
that 
$${\rm Hom}_{{\mathcal G}_{\mathcal L}}(\ell_i,\ell_j)=$$
$$\{g|_{\ell_i}: g\in {\rm Aut}(\mathcal L)\text{ is an even permutation of the lines $\mathcal L$ with } g(\ell_i)=\ell_j\}.$$
Note that there are five subgroups of ${\rm Aut}(\mathcal L)\cong S_5$ isomorphic to $S_4$; they are the subgroups of transformations mapping $\ell_i$ to itself
for each of the five lines $\ell_i$.
Thus ${\rm Hom}_{{\mathcal G}_{\mathcal L}}(\ell_i,\ell_i)$ is the subgroup of even permutations $A_4\subset S_4\subset S_5\cong{\rm Aut}(\mathcal L)$ in the subgroup $S_4$ 
corresponding to $\ell_i$.
\end{remark}

\section{The classification for the group $S_4$}

The classification for $A_4$ showed that, after normalizing the
configuration to contain
$
L_\infty,\ L_0,\ L_{I_2},
$ 
every additional line comes from a set of only eight additional candidate lines, 
corresponding to the eight 3-cycles in $A_4$. Moreover, the compatibility
condition (showing which pairs $(U,V)$ of these eight could be taken together)
reduced to a simple condition in $A_4$ (namely that $UV$ must be an element of order 2).
In particular, the vertices of the compatibility graph shown in Figure~\ref{A4CompGraph}
are $U_5,\ldots,U_{12}$ and the maximal cliques are the edges.

For $S_4$ the same strategy applies, but several new phenomena appear. In
particular, for some values of $i$
the matrix $cU_i$ can be in
$\calm_{S_4}$ for more than one value of $c$ and the maximal cliques in the compatibility graph are not just edges. Thus one should
not expect a uniqueness statement analogous to Theorem~\ref{thm:A4-classification}.
Instead, the goal is to obtain a finite list of projective equivalence classes.

\subsection{The set $\calm_{S_4}$}\label{subs:fixed copy of S4}

As was the case for the group $A_4$, for $S_4$ we first find a set $\calm_{S_4}'$ and using a compatibility graph we 
confirm that $\calm_{S_4}=\calm_{S_4}'$.

\begin{lemma}[Possible lines for $S_4$]\label{lem:S4-possible-lines}
Assume that $\operatorname{char}(K)=0$, and let 
$G\subset {\rm PGL}_2(K)$ be the subgroup isomorphic to $S_4$ given by 
$[U_1],\ldots,[U_{24}]$ for the matrices $U_i$
introduced in Example \ref{S4Example}. Let $\mathcal L=\{L_\infty,L_0,L_{M_1},\ldots,L_{M_r}\}$, $M_1=I_2$,
be a set of skew lines such that $G_{\mathcal  L}=G$. Then every matrix $M_i$ belongs to the following finite set:
\[
\calm_{S_4}'=\calm_{S_4}^{\rm ns}\cup
\left\{-I_2,\ I_2,\ \frac12 I_2,\ 2I_2\right\},
\]
where
\[
\begin{aligned}
\calm_{S_4}^{\rm ns}=\Big\{&
\pm U_2,\ \pm U_3,\ \pm U_4,\ U_5,\ldots,U_{12},\\
&U_{13},\ \frac12 U_{13},\ U_{14},\ \frac12 U_{14},\ U_{17},\ \frac12 U_{17},\ U_{19},\ \frac12 U_{19},
  \ U_{22},\ \frac12 U_{22},\ U_{24},\ \frac12 U_{24}\Big\}.
\end{aligned}
\]
In particular, $\calm_{S_4}\subseteq \calm_{S_4}'$ and 
$|\calm_{S_4}'|=30$.
Consequently for every configuration 
$\mathcal L'=\{L_\infty,L_0,L_{I_2},L_{M_2'},\ldots,L_{M_s'}\}$ with $(1)\subsetneq G_{\mathcal L'}\subseteq G$ we have $M_i'\in \calm_{S_4}'$ for each $i$.
\end{lemma}

\begin{proof}
Since $(1)\subsetneq G_{\mathcal L'}$,
at least one of the matrices $M=M_i$
must not be a scalar matrix.
Since $G_{\mathcal L}=G$, the image of $M$ in ${\rm PGL}_2(K)$ belongs to $G$.
Thus $M=cU_j$
for some $c\in K^*$ and some $j\in\{2,\ldots,24\}$.

Since $I_2=M_1$ also occurs in the configuration, the difference
$M-I_2=cU_j-I_2$
is invertible and its image in ${\rm PGL}_2(K)$ belongs to $G$. Hence there exist
$k\in\{1,\ldots,24\}$ and $\lambda\in K^*$ such that
\begin{equation}\label{Eq5.1}
cU_j-I_2=\lambda U_k. 
\end{equation}
Since $M$ and hence $U_j$ is not a scalar matrix, neither is $U_k$.
Solving these equations for the fixed representatives $U_1,\ldots,U_{24}$ gives the list of values of $c$
 in Table \ref{c-valueFig}.

\begin{table}[ht]
\centering
\[
\begin{array}{c|c|c}
\text{representative }U_j & 
\text{values of }c & \begin{matrix}\text{cycle decomposition}\\ \text{of the given elements $U_j$}\end{matrix}\\
\hline
U_2,U_3,U_4 & 1,\ -1 & (ab)(cd)\\
U_5,\ldots,U_{12} & 1 & (abc)\\
U_{13}, U_{14}, U_{17},U_{19},U_{22},U_{24} & 1,\ \frac12 & (abcd)\\
U_{15},U_{16}, U_{18}, U_{20}, U_{21}, U_{23} & \text{none} & (ab)\\
\end{array}
\]
\caption{\label{c-valueFig}
Values of $c$ such that $cU_i\in\calm_{S_4}'$.}
\end{table}

To understand the table, it is useful to observe that Equation \eqref{Eq5.1}, $cU_j-I_2=\lambda U_k$,
implies that $U_j$ and $U_k$ commute.
Thus the corresponding elements 
$u_j, u_k\in G\subset{\rm PGL}_2(\acfield)$
are nonidentity elements in an abelian subgroup of $G\cong S_4$, and hence lie in either the Klein four group $V$ or in a cyclic subgroup of $G$.
But $U_j$ and $U_k$ cannot map to $V$
since the matrices that map to $V$ are
$U_2,U_3,U_4$, and these commute 
only up to sign (for example,
$U_2U_3=U_4=-U_3U_2$). 
Thus $u_j$ and $u_k$ lie in a cyclic subgroup.
Since $u_j\neq u_k$ and neither can be the identity (since $U_j$ and $U_k$ are not scalar matrices), the cyclic subgroup which contains $u_j$ and $u_k$ must have order at least 3, and being a cyclic subgroup of $G\cong S_4$ it must have order at most 4.
(This explains the ``none" line in Table \ref{c-valueFig}, since the $u_i$ in this case are 2-cycles and so do not live in any cyclic subgroup of order more than 2.) Also, the fact that 
$U_j, U_k$ and $I_2$ map to different elements of $G$ implies
any two of them are linearly independent
and hence $U_j, U_k$ and $I_2$ span a 2-dimensional linear subspace of the vector space of $2\times2$ matrices. 
Thus there is a unique solution $(c,\lambda)$ for $cU_j-I_2=\lambda U_k$. When the cyclic subgroup has order 3, the solution is $(c,\lambda)=(1,-1)$. (This explains the second line of Table \ref{c-valueFig}.)
For each $u_j$ in a cyclic subgroup of order 4 there are two such equations; for example, $u_{13}, u_2, u_{14}$ are the nonidentity elements of a cyclic subgroup of order 4 (since $u_2=u_{13}^2, u_{14}=u_{13}^3$), and we have the equations $cU_2-I_2=\lambda U_{13}$
(which has the unique solution
$(c,\lambda)=(-1,-1)$)
and $cU_2-I_2=\lambda U_{14}$
(which has the unique solution
$(c,\lambda)=(1,-1)$). 
This explains the entry in  Table \ref{c-valueFig} for $U_2$. From $-U_2-I_2=-U_{13}$
we get $U_{13}-I_2=U_2$, and by adding
$-U_2-I_2=-U_{13}$ to $U_2-I_2=-U_{14}$
we get $\frac{1}{2}U_{13}-I_2=\frac{1}{2}U_{14}$. This explains the entry in Table \ref{c-valueFig} for $U_{13}$ (the entry for $U_{14}$ is similar).
The data in Table \ref{c-valueFig} for the other $U_i$ with two values for $c$ come from the other two cyclic subgroups of order 4.

From Table \ref{c-valueFig} we see the non-scalar matrix subset of $\calm_{S_4}'$ is
\[
\begin{aligned}
\calm_{S_4}^{\rm ns}=\Big\{&
\pm U_2,\ \pm U_3,\ \pm U_4,\ U_5,\ldots,U_{12},\\
&U_{13},\ \frac12 U_{13},\ U_{14},\ \frac12 U_{14},\ U_{17},\ \frac12 U_{17},\ U_{19},\ \frac12 U_{19},
  \ U_{22},\ \frac12 U_{22},\ U_{24},\ \frac12 U_{24}\Big\}.
\end{aligned}
\]

It remains to consider the case where
$M=cI_2$.
We always have $I_2$ in $\calm_{S_4}'$,
by assumption, so $c=1$ always occurs. Suppose now that
$c\neq 1$. Since $G_{\mathcal L}=G$ is not the identity group, $\mathcal L$ contains at least one non-scalar
matrix $N\in \calm_{S_4}^{\rm ns}$,
hence $N=c_iU_i$ for some $i>1$.
The difference $N-M=c_iU_i-cI_2$
must again be invertible and must project to an element of $G$, so
we have an equation $\frac{c_i}{c}U_i-I_2=\lambda U_j$. Thus 
$(\frac{c_i}{c},\lambda)$ is one of the solutions found before. I.e.,
$\frac{c_i}{c}=c'_i$ so $c=\frac{c_i}{c_i'}$ where $c_i, c_i'$ are values in column 2 in the same row of the table above, so for $c\neq1$ we have 
$c\in\left\{-1,\frac12,2\right\}$.
Thus the scalar matrices
which can occur, including the obligatory $I_2$, are contained in
$\left\{-I_2,\ I_2,\ \frac12 I_2,\ 2I_2\right\}$.
\end{proof}

\subsection{The compatibility graph}
Let $G\cong S_4$ be the subgroup of ${\rm PGL}_2(\mathbb C)$ whose elements are $[U_i]$, $i=1,\ldots,24$.
Let $\calm_{S_4}'$ be the finite set of matrices from
Lemma~\ref{lem:S4-possible-lines}. 
(The fact that $\calm_{S_4}'=\calm_{S_4}$ is shown in Remark \ref{rem:S4-smaller-cliques}.)
We define a simple undirected graph 
$
\Gamma_{S_4}
$
as follows.

The vertices are the matrices $M\in \calm_{S_4}'$, or
equivalently the corresponding lines $L_M$.  
Two distinct vertices $M,N$ are
joined by an edge if and only if $[M-N]=[U_i]$ 
for some $i$. Since $I_2=U_1\in \calm_{S_4}'$,
$I_2$ is a vertex of $\Gamma_{S_4}$,
and since $[M-I_2]=[U_i]$
for some $i$ for each $M\in\calm_{S_4}'$
by construction, we see $I_2$ is joined by an edge to every other vertex. So every clique in 
$\Gamma_{S_4}$ remains a clique when $I_2$ is added to it. 

\begin{lemma}\label{GCandGLlemma}
Let $\mathcal C\subset \Gamma_{S_4}$ be a
subset of two or more vertices such that $I_2\in\mathcal C$. Then $\mathcal C$ is a clique
if and only if the lines in
$$
\mathcal L=\{L_\infty,L_0\}\cup\{L_M:M\in\mathcal C\}$$
are skew and
the group $G_{\mathcal L}$ is contained in $G$. Moreover, if $\mathcal C$ is a clique, then 
$G_{\mathcal L}=G_{\mathcal C}$,
where $G_{\mathcal C}$ is the subgroup of
${\rm PGL}_2(\mathbb C)$ generated by
all $[M]$ and $[M-N]$ for all distinct $M,N\in\mathcal C$.
\end{lemma}

\begin{proof}
We have $[M]\in G$ for all $M\in\calm_{S_4}'$ by construction. 
Assume $\mathcal C$ is a clique.
Then $[M]$ and $[M-N]$ are in $G$ for all 
distinct $M,N\in \mathcal C$. Hence,
as discussed in \S\ref{Intro} (see \cite{Favacchio}), the lines in $\mathcal L$ are skew and $G_{\mathcal L}=G_{\mathcal C}$.

Conversely, if the lines in $\mathcal L$ are skew, then $G_{\mathcal L}$ is defined and if $G_{\mathcal L}\subseteq G$, then $[M-N]\in G_{\mathcal L}\subseteq G$ for all
distinct $M,N\in \mathcal C$, and hence
$\mathcal C$ is a clique.
\end{proof}

\begin{definition}
A clique $\mathcal C\subseteq\mathcal M_{S_4}'$ with $I_2\in\mathcal C$ is
called $S_4$-generating if the subgroup $G_{\mathcal C}$ of $G$ generated by $[M]$ and $[M-N]$ for all distinct $M,N\in\mathcal C$ is equal to $G$.
\end{definition}

\begin{lemma}\label{NonAbLem}
Let $\mathcal C$ be a clique in 
$\Gamma_{S_4}$ with $I_2\in \mathcal C$.
Then $G_{\mathcal C}$ is nonabelian if and only if two of the matrices in the clique do not commute.
\end{lemma}

\begin{proof}
If the matrices in $\mathcal C$
all commute, then so do they and their
differences, hence $G_{\mathcal C}$ is abelian. Now say $M,N\in{\mathcal C}$,
so $M=c_iU_i$ and $N=c_jU_j$ for some $i$ and $j$. Note that
$c_iU_i$ and $c_jU_j$ commute if and only if
$U_i$ and $U_j$ commute, and that
$[c_iU_i]=[U_i]\in{\rm PGL}_2(\mathbb C)$
(and likewise $[c_jU_j]=[U_j]$).
If $U_i$ and $U_j$ do not commute, then either their images $[U_i]$ and $[U_j]$ also do not commute (in which case $G_{\mathcal C}$ is nonabelian), or $[U_iU_j]=[U_i][U_j]=[U_j][U_i]=[U_jU_i]$.
In the latter case we must have 
$cU_iU_j=U_jU_i$ for some $c\neq1$
in which case $(U_j-I_2)U_i\neq U_i(U_j-I_2)$.
If $G_{\mathcal C}$ were abelian,
then $(U_j-I_2)U_i=bU_i(U_j-I_2)$
for some $b\neq 1$, which implies
$cU_iU_j-U_i=bU_iU_j-bU_i$ or
$(c-b)U_j=(1-b)I_2$, but then $U_j$
is scalar and so commutes with $U_i$,
contrary to assumption.
\end{proof}
\begin{table}[ht]
\centering
$$
\begin{array}{c|ccccccccc}
|\mathcal C| & 1&2&3&4&5&6&7&8\\
\hline
\#\{\text{cliques }\mathcal C\text{ with }I_2\in\mathcal C\}& 1 & 29 & 180 & 396 & 385 & 189& 63 & 9\\
\#\{\text{$S_4$-generating cliques }\mathcal C\} & 0 & 0 & 120 & 380 & 385 & 189 & 63 & 9
\end{array}
$$
\caption{\label{CliqueCountsFig}
Counts of cliques $\mathcal C\subset\Gamma_{S_4}$
by size.}
\end{table}

\begin{remark}\label{rem:S4-smaller-cliques}
We computed the graph $\Gamma_{S_4}$ and recursively found the cliques containing $I_2$ and which ones were $S_4$-generating. (By Lemma \ref{GCandGLlemma}, the cliques in $\Gamma_{S_4}$ relevant for normalized configurations are precisely the cliques containing $I_2$.) Counts of these cliques are shown in Table \ref{CliqueCountsFig}. 
The graph $\Gamma_{S_4}$ has $30$ vertices and $209$ edges. Of these 209,
there are 29 having $I_2$ as an endpoint
giving 29 cliques containing $I_2$ of size 2 as shown in
Table \ref{CliqueCountsFig},
which gives $209-29=180$ cliques containing $I_2$ and having size 3 (also as shown in the table).
The fact that there are no
$S_4$-generating cliques of sizes 1 or 2 follows from Lemma \ref{NonAbLem}.
Moreover, a maximal clique containing $I_2$ has size either $4$ (of which there is only one), $5$ (of which there are 70) or $8$ (of which there are 9).

None of the 180 cliques containing $I_2$ of size 3 is maximal. Of these 180 cliques, 120 are $S_4$-generating, 12 have group $A_4$ and the rest have groups which are cyclic abelian. By Lemma \ref{NonAbLem},
the 120 $S_4$-generating cliques of size 3
correspond to the 120 edges $\{M,N\}$ where $M$ and $N$ do not commute and where $\{M,N\}$ is not an edge in the compatibility graph for $A_4$ (given
in Figure \ref{A4CompGraph}). All other 
$S_4$-generating cliques are, by the same lemma, cliques containing both $I_2$ and one of these 120 edges. 

The unique maximal clique of size 4 is $\mathcal C=\{-I_2,\ I_2,\ \tfrac12 I_2,\ 2I_2\}$; its group $G_{\mathcal C}$ has order 1 so it is not $S_4$-generating.

Table \ref{CliqueCountsFig} shows that all cliques containing $I_2$ of size at least 5
are $S_4$-generating.

Except for the unique maximal clique containing $I_2$ of size 4 mentioned above, every clique containing $I_2$ of size 4 is contained in either a clique of size 5 or of size 8.
Two cliques containing $I_2$ of size $8$ have at most 3 vertices in common. Thus if a clique containing $I_2$ of size 4 is contained in a clique of size 8, it is contained in a unique clique of size 8. Because no
clique containing $I_2$ of size 6 or 7 is maximal, 
this also means that every clique containing $I_2$ of size 6 or 7 is contained in a unique clique of size 8.
As for cliques containing $I_2$ of size 4, there are $9\cdot \binom{7}{3}=315$ containing $I_2$ and contained in cliques of size $8$; of these, $120$ are also contained in a maximal clique of size $5$. The remaining $396-1-315=80$ are contained in maximal cliques of size $5$.
Note however, that a clique of size $4$ may be contained in two distinct maximal cliques of
size $5$. For example, the two maximal cliques
$
\{I_2,\ U_2,\ U_3,\ U_4,\ U_8\}
$
 and 
$
\{I_2,\ U_2,\ U_3,\ U_8,\ U_9\}
$ 
share 4 vertices.

The nine maximal cliques of size $8$ are:
\newline
$
\{-U_2,\ U_2,\ -U_3,\ U_3,\ -U_4,\ U_4,\ -I_2,\ I_2\},
\{U_2,\ U_6,\ U_8,\ U_9,\ U_{11},\ U_{13},\tfrac12 U_{13},\ I_2\}
$,\newline
$
\{-U_3,\ U_6,\ U_7,\ U_{10},\ U_{11},\ U_{22},\tfrac12 U_{22},\ I_2\},
\{-U_4,\ U_6,\ U_7,\ U_9,\ U_{12},\ U_{19},\tfrac12 U_{19},\ I_2\}
$,\newline
$
\{-U_2,\ U_5,\ U_7,\ U_{10},\ U_{12},\ U_{14},\tfrac12 U_{14},\ I_2\},
\{U_3,\ U_5,\ U_8,\ U_9,\ U_{12},\ U_{24},\tfrac12 U_{24},\ I_2\}
$,\newline
$
\{U_4,\ U_5,\ U_8,\ U_{10},\ U_{11},\ U_{17},\tfrac12 U_{17},\ I_2\},
\{U_{13},\ U_{14},\ U_{17},\ U_{19},\ U_{22},\ U_{24},\ I_2,\ 2I_2\}
$, \newline
and
$
\{\tfrac12 U_{13},\ \tfrac12 U_{14},\tfrac12 U_{17},
\tfrac12 U_{19},\tfrac12 U_{22},\tfrac12 U_{24},\ I_2,\tfrac12 I_2\}
$.
All of these maximal cliques of size $8$ are $S_4$-generating. Moreover,
every matrix in $\calm_{S_4}'$ is in one of these 9 cliques of size 8 and is thus in an $S_4$-generating clique, which shows that $\calm_{S_4}=\calm_{S_4}'$.
Under the
conjugation action of the fixed group $G\simeq S_4$, 
these nine cliques 
split into four conjugacy orbits, of sizes
1, 6, 1 and 1; i.e.,
the singleton orbits are represented by
the first clique of size 8 and by the last two cliques of size 8 listed above. (The last one listed corresponds in fact to the ten-line configuration from Example~\ref{S4Example}.)
The remaining six cliques listed form one conjugacy orbit.

Similarly, under the
conjugation action of the fixed group $G\simeq S_4$, the $70$ maximal $S_4$-generating cliques of size $5$ split into eight
orbits. Representatives of these orbits may be chosen as shown in Table \ref{ConjClassFig}.
\qed
\end{remark}

\begin{table}[ht]
\centering
$$
\begin{array}{c|l}
\text{orbit size} & \text{representative}\\
\hline
2 &
\{I_2,U_5,U_7,U_9,U_{11}\}\\
8 &
\{I_2,U_6,U_7,U_9,U_{11}\}\\
8 &
\{I_2,-U_2,U_3,U_4,U_5\}\\
12 &
\{I_2,U_2,U_3,U_8,U_9\}\\
8 &
\{I_2,U_{10},\tfrac12U_{14},\tfrac12U_{17},\tfrac12U_{22}\}\\
8 &
\{I_2,U_6,U_{13},U_{19},U_{22}\}\\
12 &
\{I_2,U_5,U_8,\tfrac12U_{17},\tfrac12U_{24}\}\\
12 &
\{I_2,U_6,U_9,U_{13},U_{19}\}\\
\hline
\llap{{\rm Total = }}70 & \\
\end{array}
$$
\caption{\label{ConjClassFig}
Conjugacy class representatives for maximal cliques in $\Gamma_{S_4}$ containing $I_2$ of size 5.
}
\end{table}
The next result shows, over the complex numbers, that the extremal cases of 
skew lines whose group is $S_4$ are unique up to projective equivalence. In particular,
5 is the least number of skew lines whose group can be $S_5$ since the group for any 4 skew lines
is abelian. The cases of sets of skew lines $\mathcal L=\{L_\infty, L_0, L_{I_2}, L_{M_4},L_{M_5}\}$
with group $S_4$ correspond to the 120 $S_4$-generating cliques $\{I_2,M_4,M_5\}$ of Table \ref{CliqueCountsFig}.
A direct examination of all $S_4$-generating cliques shows that every $S_4$-generating clique contains
an $S_4$-generating clique of size 3. Thus the minimal $S_4$-generating cliques are those of size 3;
i.e., if some set $\mathcal L'$ of skew lines containing $\{L_\infty, L_0, L_{I_2}\}$ has group $S_4$ but no subset of 
$\mathcal L'$ containing $\{L_\infty, L_0, L_{I_2}\}$ has group $S_4$, then $\mathcal L'$ has 5 lines.
And we saw above that  maximal sets of skew lines with group $S_4$ have either 7 or 10 lines. 
For nonextremal cases (such as sets of 6, 8 or 9 lines, and nonmaximal sets of 7 lines) we have not 
computed the number of projective equivalence classes, but, as Example \ref{S4Example} shows,
there are at least two projective equivalence classes of 6 skew lines having group $S_4$.

\begin{theorem}\label{ProjEquivThm}
All sets $\mathcal L$ of $s$ skew lines in $\PP^3_{\mathbb C}$ with group $S_4$ are projectively equivalent if $s=5$ or if $s=10$. 
The same holds for $s=7$ for sets $\mathcal L$ with group $S_4$ not contained in any larger set with group $S_4$.
\end{theorem}

\begin{proof}
We use the same strategy for $s=5$, $s=7$ and $s=10$.
Consider first the case of $s=10$. By Remark \ref{rem:S4-smaller-cliques}, there are
9 sets of 10 skew lines with group $S_4$. Again by Remark \ref{rem:S4-smaller-cliques}, 
the first and last two of the listed 
$S_4$-generating cliques of size 8 are conjugate under a $U_i$ only to 
themselves while the other 6 are all conjugate to each other.

Thus every set of 10 skew lines with group $S_4$ is projectively equivalent to those corresponding to one of the
first two or one of the last two $S_4$-generating cliques of size 8, each of which is of the form $\mathcal L=\{L_\infty,L_0,L_{M_1},\ldots,L_{M_8}\}$ where $M_1=I_2$ and each matrix $M_i$, $i>1$, is $c_jU_j$ for some scalar $c_j$. For exactly one of these 4
(namely the case that the matrices $M_i$ are
$-U_2,U_2,-U_3,U_3,-U_4,U_4,-I_2,I_2$)
the matrices mod scalars all have order at most 2.

Suppose for each of the other 3 conjugacy classes that we  can find a projective transformation $\Lambda$ that takes the 10 lines $\mathcal L$
corresponding to a representative of that conjugacy class to another set of 10 lines $\mathcal L'=\{L_\infty, L_0, L_{I_2}, L_{M_2'},\ldots, L_{M_8'}\}$ 
where the matrices $M_i'$ all have order (mod scalars) at most 2. By Proposition \ref{A_GProp}(4) there is a $\psi_B\Lambda$ giving 
lines $L_\infty, L_0, L_{I_2}, L_{BM_2'B^{-1}},\ldots, L_{BM_8'B^{-1}}$ but now the group is the same $S_4$ in
${\rm PGL}_2(\mathbb C)$ as for $-U_2,U_2,-U_3,U_3,-U_4,U_4,-I_2,I_2$, so conjugating by an appropriate element of 
$\mathcal A_G=S_4$ would send $\mathcal L'$ to $\{L_\infty, L_0, L_{I_2}, -U_2, U_2, -U_3, U_3,$ $-U_4, U_4,\allowbreak -I_2, I_2\}$,
thus showing that a set of 10 lines coming from the first conjugacy class is projectively equivalent to sets coming from the other conjugacy classes.
There are $720=10(9)8$ ways to choose lines $L_a,L_b,L_c\in \mathcal L$ to send, in order, to $L_\infty, L_0, L_{I_2}$.
For each one a computation shows one can in fact find a $\Lambda$ that implements it, and using $\Lambda$ we can find the matrix $M$ 
giving the image $L_M$ for each of the other 7 lines and check the order of the image of $M$ in ${\rm PGL}_2(\mathbb C)$.
It turns out for each of the 720 ways for each of the 3 conjugacy classes,
there are 80 for which the 7 matrices all have order at most 2, thereby showing that sets of 10 skew lines with group $S_4$
are projectively equivalent.

The case of maximal sets of 7 lines with group $S_4$ is similar. Now there are 70 sets 
of 7 lines comprising 8 conjugacy classes, listed in 
Table \ref{ConjClassFig}. For exactly one of these conjugacy classes
the matrices all have order (mod scalars) at most 3 with exactly one of order 3 (this is the third one listed in the table). For each of the 7 other
conjugacy classes there are $3!\binom{7}{3}=210$
choices of $L_a,L_b,L_c$ to map to $L_\infty,L_0,L_{I_2}$.
Of these, there are always 24 (although we need only one) for which the new matrices all have order (mod scalars) at most 3 with exactly one of order~3.

Finally, consider the minimal case: sets of 5 lines
with group $S_4$. By Table~\ref{CliqueCountsFig} there are 120 such sets,
corresponding to cliques containing $I_2$ of size 3. 
These divide into seven conjugacy classes with respect to conjugation by elements $U_i$. Representatives of each class are: 
$\{U_1,-U_2,U_3\}$, 
$\{U_1,-U_2,U_5\}$,
$\{U_1,U_5,U_7\}$,
$\{U_1,U_5,U_{14}\}$,
$\{U_1,U_5,\frac12U_{14}\}$, 
$\{U_1,U_{13},U_{17}\}$ and
$\{U_1,\frac12U_{13},\frac12U_{17}\}$. The only case where both matrices have order 2 (mod scalars)
is the first one. For each of the other 6, 
there are $3!\binom{5}{3}=60$ ways to pick 3 of the 5 lines
to send to $L_\infty, L_0, L_{I_2}$, 6 of which
give a pair of matrices for the remaining two lines, where both matrices have order 2. 
Thus there is a single projective equivalence class for sets of 5 lines 
whose group is~$S_4$.
\end{proof}

By Remark \ref{rem:S4-smaller-cliques}, there are 70 maximal cliques of size $5$ giving configurations of seven skew lines
whose associated group is $S_4$. Hence the $S_4$ case is not necessarily obtained only by taking
subconfigurations of the ten-line examples coming from maximal cliques of size $8$.  We illustrate this phenomenon in Example \ref{ex:S4-seven-lines} with one
representative clique 
and record explicitly how the group $S_4$ is generated. We also note that, for such a configuration,  no four of the seven lines lie on a quadric surface; in Remark \ref{rem:quadrics-example-25}, we show that this property does not always hold in analogous situations.
\begin{example}\label{ex:S4-seven-lines}
Consider the maximal clique 
$$
\mathcal C=\{I_2,\ U_2,\ U_3,\ U_4,\ U_8\}
$$ 
of size $5$ in $\Gamma_{S_4}$, and the corresponding configuration
of seven skew lines
$$
L_{\mathcal C}
=
\{L_\infty,L_0,L_{I_2},L_{U_2},L_{U_3},L_{U_4},L_{U_8}\}.
$$
This is one of the maximal seven-line configurations occurring in
Theorem~\ref{ProjEquivThm}. We record it explicitly in order to show
how the associated group $S_4$ is generated in this normalized model.

The matrices $U_2,U_3,U_4$ represent the Klein four subgroup of
$A_4$, and $U_8$ has order $3$. 
Hence the images of $U_2,U_3,U_4,U_8$ generate the subgroup $A_4$.
Moreover, conjugation by $U_8$ cyclically
permutes the three nontrivial elements of this Klein four subgroup: 
$$
U_8U_2U_8^{-1}\equiv U_3,\quad
U_8U_3U_8^{-1}\equiv U_4,\quad
U_8U_4U_8^{-1}\equiv U_2.
$$

Since $\mathcal C$ is a clique in $\Gamma_{S_4}$, all pairwise
differences of matrices in $\mathcal C$ project to elements of the fixed
copy of $S_4$. For this particular clique, up to nonzero scalar multiples,
we have 
$$
I_2-U_2\equiv U_{14},\quad
I_2-U_3\equiv U_{22},\quad
I_2-U_4\equiv U_{19},\quad
I_2-U_8\equiv U_7,\quad
U_2-U_3\equiv U_{23}, 
\ \text{and}
$$
$$U_2-U_4\equiv U_{18},\quad
U_2-U_8\equiv U_5,\quad
U_3-U_4\equiv U_{16},\quad
U_3-U_8\equiv U_{11},\quad
U_4-U_8\equiv U_9.
$$
In particular, the difference $I_2-U_2$ gives an element in the odd coset
of $A_4$ in~$S_4$. Therefore the group generated by the images of the
matrices in $\mathcal C$ and of their pairwise differences contains $A_4$
and one element outside $A_4$. Hence it is the full group $S_4$.
\end{example}

\begin{remark}\label{rem:quadrics-example-25}
For the seven-line configuration $L_{\mathcal C}$ of
Example~\ref{ex:S4-seven-lines}, a direct computation shows that no four
of its lines are contained in a quadric surface. 

The ten-line configuration
$$
L=\{L_\infty,L_0,L_{M_1},\ldots,L_{M_8}\}
$$
introduced in Example~\ref{S4Example} behaves differently. In this case
there are exactly ten quadrics containing four of the ten lines, and no
quadric contains five or more of them. More precisely, the ten lines split into five pairs
$$
\{L_\infty,\ L_{M_2}\}, \ 
\{L_0,\ L_{M_1}\}, \ 
\{L_{M_3},\ L_{M_4}\}, \ 
\{L_{M_5},\ L_{M_6}\}, \ 
\{L_{M_7},\ L_{M_8}\}.
$$
The four-line quadrics are precisely the quadrics containing the union of
any two of these pairs. Thus the incidence structure of such quadrics is
the complete graph on these five pairs.
\end{remark}

\end{document}